\documentclass[11pt]{article}

\usepackage[a4paper,margin=28mm]{geometry}
\usepackage[T1]{fontenc}
\usepackage{lmodern}
\usepackage{microtype}
\usepackage{amsmath,amssymb,amsthm,mathtools}
\usepackage{mathrsfs}
\usepackage{booktabs}
\usepackage{array}
\usepackage{tabularx}
\usepackage{aliascnt}
\usepackage{enumitem}
\usepackage{xcolor}
\usepackage[colorlinks=true,linkcolor=blue!55!black,citecolor=blue!55!black,
  urlcolor=blue!55!black,
  pdftitle={The normalized approximation function for multiple zeta-star values}]{hyperref}
\usepackage[nameinlink,noabbrev]{cleveref}
\usepackage{authblk}
\allowdisplaybreaks[2]
\numberwithin{equation}{section}

\makeatletter
\let\@date\@empty
\makeatother

\newtheorem{theorem}{Theorem}[section]
\newaliascnt{proposition}{theorem}
\newtheorem{proposition}[proposition]{Proposition}
\aliascntresetthe{proposition}
\newaliascnt{lemma}{theorem}
\newtheorem{lemma}[lemma]{Lemma}
\aliascntresetthe{lemma}
\newaliascnt{corollary}{theorem}
\newtheorem{corollary}[corollary]{Corollary}
\aliascntresetthe{corollary}
\newaliascnt{conjecture}{theorem}
\newtheorem{conjecture}[conjecture]{Conjecture}
\aliascntresetthe{conjecture}

\theoremstyle{definition}
\newaliascnt{definition}{theorem}
\newtheorem{definition}[definition]{Definition}
\aliascntresetthe{definition}

\theoremstyle{remark}
\newaliascnt{remark}{theorem}

\aliascntresetthe{remark}

\crefname{theorem}{Theorem}{Theorems}
\crefname{proposition}{Proposition}{Propositions}
\crefname{lemma}{Lemma}{Lemmas}
\crefname{corollary}{Corollary}{Corollaries}
\crefname{conjecture}{Conjecture}{Conjectures}
\crefname{definition}{Definition}{Definitions}
\crefname{remark}{Remark}{Remarks}

\newcommand{\R}{\mathbb{R}}
\newcommand{\NN}{\mathbb{N}}
\newcommand{\Sset}{\mathcal{S}}
\newcommand{\Tset}{\mathcal{T}}
\newcommand{\Dset}{\mathcal{D}}

\newcommand{\Lnorm}{\mathscr{L}}
\newcommand{\wt}{\operatorname{wt}}
\newcommand{\dep}{\operatorname{dep}}
\newcommand{\dist}{\operatorname{dist}}

\newcommand{\ones}[1]{\{1\}^{#1}}

\title{The normalized approximation function for multiple zeta-star values}
\author{Jiangtao Li}

\affil{School of Mathematics and Statistics, HNP-LAMA, Central South University, China}

\affil{\texttt{lijiangtao@csu.edu.cn}}

\begin{document}
\maketitle

\begin{abstract}
Motivated by the classical Lagrange spectrum and the reciprocal formulation of the Lagrange spectrum in continued-fraction theory, we introduce a normalized approximation function $\mathcal{N}(\alpha)$ for approximation by multiple zeta-star values. At each depth, the approximation error is minimized over all admissible indices and normalized by the binary scale determined by their weights. The function $\mathcal{N}(\alpha)$ is then obtained by taking the limit inferior as the depth tends to infinity. 

Using the  order structure of multiple zeta-star values, we establish the regularity and generic behavior of this function. We prove that the depthwise approximation functions are upper semicontinuous, that $\mathcal{N}$ is Borel measurable, and that its zero set is a dense $G_\delta$ subset of $(1,+\infty)$.  We also derive a natural-prefix approximation estimate showing that a large next digit produces an exceptionally good normalized approximation. Combined with the author’s metric results on Diophantine approximation of multiple zeta-star values, this gives a full-measure theorem under a logarithmically reinforced divergence condition and, in particular, proves that $\mathcal{N}(\alpha)=0$ for Lebesgue almost every $\alpha>1$. Finally, we  give some basic properties of  the image of the normalized approximation function and show that the image $\mathrm{Im}\; \mathcal{N}$ is dense on the extended half-line $[0,+\infty]$.

\end{abstract}

\noindent\textbf{Keywords.}
Multiple zeta-star values; normalized Diophantine approximation;
Lagrange spectrum; metric number theory.

\medskip
\noindent\textbf{2020 Mathematics Subject Classification.}
Primary 11M32; Secondary 11J70, 11J83.

\begingroup
\makeatletter
\renewcommand{\thefootnote}{}
\renewcommand{\@makefntext}[1]{%
  \noindent\hspace*{0em}#1%
}
\footnotetext{%
  Project funded by the National Natural Science Foundation of China
  (Grant No.12571009) and the Natural Science Foundation of Hunan
  Province, China (Grant No.2026JJ40003).%
}
\makeatother
\endgroup

\section{Introduction}\label{sec:intro}

\subsection{The theory of multiple zeta-star values}

For an admissible finite index
\[
 \mathbf{k}=(k_1,\ldots,k_r),
 \qquad k_1\ge2,\quad k_2,\ldots,k_r\ge1,
\]
the multiple zeta value and the multiple zeta-star value are defined,
respectively, by
\begin{equation}\label{eq:mzsv}
 \zeta(\mathbf{k})
 :=\sum_{n_1>\cdots>n_r\ge1}
 \frac{1}{n_1^{k_1}\cdots n_r^{k_r}},
 \qquad
 \zeta^\star(\mathbf{k})
 :=\sum_{n_1\ge\cdots\ge n_r\ge1}
 \frac{1}{n_1^{k_1}\cdots n_r^{k_r}}.
\end{equation}
The admissibility condition $k_1\ge2$ guarantees convergence.  The weight
and depth of $\mathbf{k}$ are
\[
 \wt(\mathbf{k})=k_1+\cdots+k_r,
 \qquad
 \dep(\mathbf{k})=r.
\]
Thus the star version differs from the ordinary multiple zeta value by
allowing equality between consecutive summation variables.  Equivalently,
$\zeta^\star(k_1,\ldots,k_r)$ is the sum of the ordinary multiple zeta
values obtained by replacing each comma between adjacent entries by either a
comma or a plus sign, with a plus sign merging the two entries.  For example,
\[
 \zeta^\star(a,b,c)
 =\zeta(a,b,c)+\zeta(a+b,c)+\zeta(a,b+c)+\zeta(a+b+c).
\]
  Identities for these values,
their algebraic comparison with ordinary multiple zeta values, and their
interpolation have been developed extensively; see
\cite{HoffmanMHS,IKOO,OZ,Yamamoto,Zagier}.

Define
\[
\mathcal{Z}^\star
 :=\big{\{}\zeta^\star(\mathbf{k})\;\big{|}\;
 \mathbf{k}\text{ is a finite admissible index}\big{\}}.
\]
Whereas much of the classical theory concerns algebraic relations among
these periods, the topology and metric distribution of $\mathcal{Z}^\star$ in the real
line have only recently begun to be studied systematically.  In
\cite{LiTopology}, the author introduced a total order
on admissible indices, proved that $\mathcal{Z}^\star$ is a countable dense subset of
$(1,+\infty)$, and calculated Hausdorff dimensions for several symbolically
defined subsets.  Hirose, Murahara, and Onozuka independently developed
multiple zeta-star values for indices of infinite length and established
closely related differentiability results \cite{HMOInfinite}.

Let
\begin{equation}\label{eq:Tset}
 \Tset
 :=\Bigl\{(k_1,k_2,\ldots)\in\NN^{\NN}:
 k_1\ge2,\ k_i\ge1\ (i\ge2),\quad
 k_1=2\Rightarrow k_s\ge2\text{ for some }s\ge2\Bigr\}.
\end{equation}
Thus $\Tset$ consists of all infinite admissible sequences beginning with an
entry at least $2$, except for the single sequence $(2,1,1,\ldots)$.  The
exceptional sequence is excluded because its finite-prefix values do not
represent a finite point of $(1,+\infty)$.  For
${\bf k}=(k_1,k_2,\ldots)\in\Tset$,  there is an order-compatible bijection 
\[
 \eta:\Tset\rightarrow(1,+\infty)
\]
 \begin{equation}\label{eq:eta}
 \eta({\bf k})
 :=\lim_{r\to\infty}\zeta^\star(k_1,\ldots,k_r).
\end{equation}
The limit exists for every $ {\bf k}\in\Tset$.
This order-compatible bijection  was found by the author \cite{LiTopology}, Hirose, Murahara and Onozuka \cite{HMOInfinite} independently.  We call
this bijection the {\bf{zeta-star correspondence}}.  Consequently, every
real number $\alpha>1$ has a unique infinite sequence
$\eta^{-1}(\alpha)$, while its finite prefixes provide a canonical sequence
of approximants.  The corresponding basic intervals are symbolic cylinders,
and their binary coordinates have scale comparable to $2^{-K_r}$, where
$K_r=k_1+\cdots+k_r$.

Subsequent work has explored several directions opened by the zeta-star
correspondence.  The author studied rational deformations and their derived sets,
metric structure, and relations with Cantor-type sets \cite{LiRational}; established a
Diophantine criterion and zero-one laws for approximation by multiple
zeta-star values
\cite{LiDA}; and used divided differences to construct complex-analytic
interpolations of finite multiple star harmonic sums, finite zeta-star
correspondences for real parameters, and injectivity conjectures for complex
variations \cite{LiComplex}.  The author and Yang investigated arithmetic sums and
products of infinite multiple zeta-star values with restricted indices and
proposed conjectures concerning algebraic points and the arithmetic geometry
of the resulting sumsets and product sets \cite{LY}.  Kamano \cite{kama} studied the order structures of multi-polylogarithms. Taken
together, these works suggest that the zeta-star correspondence is not merely
an order-theoretic parametrization: it also supports Diophantine, fractal,
arithmetic, and analytic questions analogous to those arising from continued
fractions and binary expansions.

\subsection{Reciprocal Lagrange motivation and normalized approximation function}

The present paper concerns the quantitative approximation problem
arising from the zeta-star correspondence.  Since $\mathcal{Z}^\star$ is dense,
unnormalized approximation by its elements is automatic; a meaningful
problem requires a complexity scale.  Both the weight of an index and the
binary length of its symbolic cylinder single out the normalization
$2^{\wt(\mathbf{k})}$.  This is reminiscent of continued-fraction
approximation, where denominators of convergents determine the natural
quadratic scale.

For $y\in\R$, write
\[
 \lVert y\rVert:=\min_{\ell\in\mathbb{Z}}|y-\ell|.
\]
For an irrational number $x$, the reciprocal Lagrange constant can be written
as
\[
 \xi(x):=\liminf_{n\to\infty}n\lVert nx\rVert,
\]
and the reciprocal of this quantity is governed by Perron's formula; see
\cite{CF,Matheus}.  The comparison motivating the present
normalization is geometric.  If
$\alpha=\eta(k_1,k_2,\ldots)$ and $K_r=k_1+\cdots+k_r$, then the finite prefix
$\zeta^\star(k_1,\ldots,k_r)$ is the natural approximant, while
$2^{-K_r}$ is the scale of its binary cylinder.

Let $\Sset_r$ be the set of admissible indices of depth $r$.  For $\alpha>1$,
define
\begin{equation}\label{eq:Delta-definition}
 \Delta_r(\alpha)
 :=\inf_{\mathbf{k}\in\Sset_r}
 2^{\wt(\mathbf{k})}
 \bigl|\alpha-\zeta^\star(\mathbf{k})\bigr|,
\end{equation}
and
\begin{equation}\label{eq:N-definition}
 \mathcal{N}(\alpha):=\liminf_{r\to\infty}\Delta_r(\alpha)
 \in[0,+\infty].
\end{equation}
The infimum is taken at fixed depth, while the weights are allowed to vary.
Thus $\mathcal{N}(\alpha)$ measures the best asymptotic error at each depth
after normalization by the intrinsic binary scale.  It is a
reciprocal-Lagrange-type observable rather than a direct identification with
the classical Lagrange spectrum.

For every $C>0$, the definition  of $\mathcal{N}(\alpha)$ is equivalently expressed as follows:
\[
 \mathcal{N}(\alpha)<C
 \quad\Leftrightarrow\quad
 \left|\alpha-\zeta^\star(\mathbf{k})\right|
 <\frac{C-\varepsilon}{2^{\wt(\mathbf{k})}}
\]
for some $\varepsilon\in(0,C)$ and infinitely many admissible indices
$\mathbf{k}$ whose depths tend to $+\infty$.  Hence
$\mathcal{N}(\alpha)$ quantifies \textbf {approximation of $\alpha$ by  multiple
zeta-star values at the natural binary scale}.

For a function $g:\NN\to[1,+\infty)$, define the limsup set
\begin{equation}\label{eq:Bg-intro}
 \mathcal{B}_g
 :=\left\{\alpha>1:\Delta_r(\alpha)<2^{-g(r)}
 \text{ for infinitely many }r\right\}.
\end{equation}
This formulation leads to Borel--Cantelli and Khintchine-type questions: one
asks when the normalized targets are hit only finitely often, when they are
hit for almost every $\alpha$, and how the answer depends on the interaction
between depth, weight, and symbolic distortion.

\subsection{Main results}

The main results developed in this paper are:
\begin{theorem}\label{thm:consolidated-main}
The following statements hold.
\begin{enumerate}[label=(\roman*),leftmargin=2.6em]
\item For each $r\ge1$, the map $\alpha\mapsto\Delta_r(\alpha)$ is upper
semicontinuous, and $\alpha\mapsto\mathcal{N}(\alpha)$ is Borel measurable.
Moreover,
\[
 \{\alpha>1:\mathcal{N}(\alpha)=0\}
\]
is a dense $G_\delta$ subset of $(1,+\infty)$.

\item If $\alpha=\eta(k_1,k_2,\ldots)$ and $K_r=k_1+\cdots+k_r$, then there
exist $C(\alpha)>0$ and $r_0(\alpha)$ such that, whenever
$r\ge r_0(\alpha)$ and $k_{r+1}\ge2$,
\begin{equation}\label{eq:main-prefix}
 2^{K_r}\left|\alpha-\zeta^\star(k_1,\ldots,k_r)\right|
 \le C(\alpha)r\,2^{1-k_{r+1}}.
\end{equation}

\item If $\sum_{r\ge1}2^{-g(r)}<+\infty$, then $\mathcal{B}_g$ has Lebesgue
measure zero.  If
\begin{equation}\label{eq:main-reinforced}
 \sum_{r=4}^{\infty}\frac{2^{-g(r)}}{r\log_2 r}=+\infty,
\end{equation}
then the complement of $\mathcal{B}_g$ in $(1,+\infty)$ has Lebesgue measure
zero.  In particular, $\mathcal{N}(\alpha)=0$ for Lebesgue almost every
$\alpha>1$.
\end{enumerate}
\end{theorem}

\begin{theorem}\label{imn} The image of $\mathcal{N}$ satisfies the following properties.
\begin{enumerate}[label=(\roman*),leftmargin=2.5em]\item The finite part of the image of $\mathcal{N}$ is dense in the nonnegative half-line:
\[
 \overline{\mathrm{Im}(\mathcal{N})\cap\mathbb{R}}=[0,+\infty).
\]
\item For every integer $n\ge 2$,
\[
\mathcal{N}(n)=n.
\]
\item If $\mathbf{k}=(k_1,\ldots,k_r)$ is finite and admissible, then
\begin{equation}\label{eq:main-finite-value}
 \mathcal{N}\bigl(\zeta^\star(\mathbf{k})\bigr)
 =2^{\wt(\mathbf{k})}
 \sum_{n_1\ge\cdots\ge n_r\ge2}
 \frac{n_r-1}{n_1^{k_1}\cdots n_r^{k_r}}.
\end{equation}
This value equals $+\infty$ exactly for
$\mathbf{k}=(2,\ones{r-1})$; otherwise it is finite and strictly larger than
$1$.  Consequently, every nonempty open interval contains points at which
$\mathcal{N}(\alpha)=0$ and finite multiple zeta-star points at which
$\mathcal{N}(\alpha)>1$ or $\mathcal{N}(\alpha)=+\infty$.  In particular,
$\alpha\mapsto\mathcal{N}(\alpha)$ is nowhere continuous when
$[0,+\infty]$ is endowed with its order topology.
\item For every $\alpha>1$,
\[
\mathcal N(\alpha)=+\infty
\quad\Leftrightarrow\quad
\alpha=\zeta^\star(2,\{1\}^{s-1})\;\mathrm{for\;some}\;s\geq1.
\]
\end{enumerate}
\end{theorem}

\subsection{Further remarks}

For Lebesgue almost every irrational $x$, one has $\xi(x)=0$.  The classical
Lagrange spectrum is
\[
 \mathcal{L}_{\mathrm{Lag}}
 :=\left\{\frac{1}{\xi(x)}:
 0<\xi(x)<+\infty,\ x\in\R\setminus\mathbb{Q}\right\}.
\]
Its theory is closely connected with continued fractions, binary quadratic
forms, and the Markov spectrum.  Its lower part has a remarkable arithmetic
structure: the spectrum begins at $\sqrt{5}$, contains a discrete sequence of
values below $3$ related to Markov numbers, and has $3$ as its first
accumulation point.  At larger values the structure becomes increasingly
complicated, but Hall's theorem shows that the Lagrange spectrum eventually
contains an entire half-line.  These features make the Lagrange spectrum an
important meeting point of number theory, dynamical systems, and fractal
geometry; see \cite{CF,Hall,Matheus, Mor}.  

Inspired by the definition of Lagrange spectrum, we give the following definition.
\begin{definition}
Define the approximation spectrum of multiple zeta-star values as 
\[
\mathcal{L}_{\mathrm{star}}:=\left\{\frac{1}{\mathcal{N}(\alpha)}:
 0<\mathcal{N}(\alpha)<+\infty,\ \alpha \in(1,+\infty)\right\}.
\]
\end{definition} 
From the main results of Hall \cite{Hall}, the Lagrange spectrum $\mathcal{L}_{\mathrm{Lag}}$ contains the half-line $[6,+\infty)$. Subsequent work identified the maximal Hall ray. For more details about the Hall ray, see \cite{cus,Mor}. For the spectrum $\mathcal{L}_{\mathrm{star}}$, we propose the following conjectures:
\begin{conjecture}
There is a $c>0$ such that $(c,+\infty)\subseteq  \mathcal{L}_{\mathrm{star}}$.
\end{conjecture}
\begin{conjecture}
The complementary set $$(0,+\infty)- \mathcal{L}_{\mathrm{star}}$$ is non-empty.
\end{conjecture}

The paper is organized as follows.  \Cref{sec:preliminaries} records the
theory of zeta-star correspondence  and interval geometry.  \Cref{sec:topology} establishes
regularity and the residual-zero property.  The metric estimates and almost-everywhere theorem are proved next. \Cref{sec:integers} computes the values $\mathcal{N}(\alpha)$ at some special points. \Cref{sec:image} develops the binary interpretation and the closure of the image. The final section summarizes the main results and proposes the remaining
questions concerning the image and level sets of $\mathcal{N}(\alpha)$.

\section{Binary expansion and multiple zeta-star values}\label{sec:preliminaries}

In this section, we will review the order structure of multiple zeta-star values and zeta-star correspondence. The references are $\cite{LiDA,LiTopology}$.

\subsection{Infinite indices and the zeta-star correspondence}

The infinite admissible set $\Tset$, the finite indices set $\mathcal{Z}^\star$, and the
zeta-star correspondence $\eta$ were introduced in  \Cref{sec:intro}.  We
now record the binary coordinate and  other notations used throughout
the proofs.

Define the binary coordinate by 
\[
\beta: \mathcal{T}\rightarrow \left(0,\frac{1}{2}\right),\]
\begin{equation}\label{eq:beta-infinite}
 \beta({\bf k})
 =\sum_{j=1}^{+\infty}2^{-K_j}, \quad K_j=k_1+\cdots+k_j,\quad \forall\, j\geq 1.
 \end{equation}
There is a continuous strictly increasing bijection
\begin{equation}\label{eq:Phi-map}
 \Phi:\left[0,\frac{1}{2}\right]\longrightarrow[1,+\infty],
\end{equation}
where $[1,+\infty]$ carries its order topology, with
$\Phi(0)=1$, $\Phi\left(\frac{1}{2}\right)=+\infty$, and
\begin{equation}\label{eq:Phi-eta}
 \Phi\bigl(\beta({\bf k})\bigr)
 =\eta({\bf k}).
\end{equation}
This normalization is essentially the inverse of the function \[\tau: (1,+\infty)\rightarrow \left( 0,\frac{1}{2}\right),\]
\[
\tau (\alpha)=\beta(\eta^{-1}(\alpha))
\]
 defined by the author in Section $4$, \cite{LiTopology}.  Hirose, Murahara, and  Onozuka \cite{HMOInfinite} studied the differentiability  of the function 
    \[Z^\star(x)=\Phi\left(\frac{x}{2}\right)\]
 on $(0,1)$. 
 
 For $x\in (0,1)$, we write a binary expansion as
\[
 x=\sum_{j\ge1}a_j2^{-j},
 \qquad a_j\in\{0,1\}.
\]
When the expansion contains infinitely many zeros and ones, the derivative
formulas in Theorem $4.7$, Theorem $4.8$ of Hirose, Murahara and Onozuka  \cite{HMOInfinite}, written in the normalization
\eqref{eq:beta-infinite} form, are
\begin{equation}\label{eq:derivative-series}
(Z^\star)^\prime(x)=\frac{1}{2} \Phi'\left(\frac{x}{2}\right)
 =\frac12+\sum_{d=1}^{+\infty}
 a_d2^dC_d(a_1,\ldots,a_{d-1}),
\end{equation}
where
\begin{equation}\label{eq:Cd}
 C_d(a_1,\ldots,a_{d-1})
 :=\sum_{m_1\ge\cdots\ge m_d\ge3}
 \frac{a_1^{m_1-m_2}\cdots a_{d-1}^{m_{d-1}-m_d}}
 {m_1^2m_2\cdots m_d},
\end{equation}
with the convention $0^0=1$.  At a point with a terminating binary expansion, the
corresponding finite formula gives the right derivative.

The following proposition is essentially another  reformulation of the derivative formula  \eqref{eq:derivative-series}  of Hirose, Murahara and Onozuka.
\begin{proposition}\label{prop:der}One has 
\begin{equation}\label{der}
\Phi^\prime(y)= \lim_{r\rightarrow+\infty} \sum_{n_1\geq \cdots \geq n_r\geq 2}      \left(\frac{2}{n_1}\right)^{k_1}\cdots  \left(\frac{2}{n_r}\right)^{k_r}.
\end{equation}
Here \[y=\beta({\bf k})=\sum_{r=1}^{+\infty} \frac{1}{2^{k_1+\cdots+k_r}}, \quad {\bf k}\in \mathcal{T}\]
 and $y$ is not a dyadic rational number in $\left(0, \frac{1}{2}\right)$.
 For \[y=\beta^f({\bf k})=\sum_{s=1}^{r} \frac{1}{2^{k_1+\cdots+k_s}}, \quad k_1\geq 2, k_2,\cdots,k_r\geq 1,\] 
 one has 
 \begin{equation}\label{eq:fi}
 \partial_{+}\Phi(y)=\sum_{n_1\geq \cdots \geq n_r\geq 2}      \left(\frac{2}{n_1}\right)^{k_1}\cdots  \left(\frac{2}{n_r}\right)^{k_r}.
 \end{equation}
 \end{proposition}
 \begin{proof}
For  \[y=\beta({\bf k})=\sum_{r=1}^{+\infty} \frac{1}{2^{k_1+\cdots+k_r}}, \quad {\bf k}\in \mathcal{T}\]
 and $y$ is not a dyadic rational number in $\left(0, \frac{1}{2}\right)$, define 
 \[
 y=\sum_{j\geq 1} b_j2^{-j},\quad b_j\in\{0,1\}.
 \]
 Clearly, $b_1=0$ and 
 \begin{equation}\label{eq:2m}
  2y=\sum_{j\geq 1} b_{j+1}2^{-j},\quad b_j\in\{0,1\}.
   \end{equation}
 By the derivative formula \eqref{eq:derivative-series} and \eqref{eq:2m}, it follows that 
 \[
 \begin{split}
 &\;\;\;\; \Phi^\prime(y)\\
 &=  2 (Z^\star)^\prime(2y)                       \\
 &=2\left[\frac{1}{2}+      \sum_{d=1}^{+\infty} b_{d+1} 2^{d}  \left(    \sum_{m_1\ge\cdots\ge m_d\ge3}
 \frac{b_2^{m_1-m_2}\cdots b_{d}^{m_{d-1}-m_d}}
 {m_1^2m_2\cdots m_d}       \right)    \right]  \\
 &=1+\sum_{d=1}^{+\infty} b_{d+1}2^{d+1} \left(    \sum_{m_1\ge\cdots\ge m_d\ge3}
 \frac{b_2^{m_1-m_2}\cdots b_{d}^{m_{d-1}-m_d}}
 {m_1^2m_2\cdots m_d}       \right)  . \end{split}
 \]
 Since 
 \[
 y=\sum_{r=1}^{+\infty} \frac{1}{2^{k_1+\cdots+k_r}}=\sum_{j=1}^{+\infty} \frac{b_j}{2^j}, \quad b_j\in\{0,1\},\]
 we have 
 \begin{equation}\label{eq:01}
 b_j=\begin{cases}
 0,&j\notin \{k_1,\cdots, k_1+\cdots+k_r,\cdots\}; \\
 1,& j\in \{k_1,\cdots, k_1+\cdots+k_r,\cdots\}.    \\
 \end{cases}
 \end{equation}
 Denote by 
 \[
 K_r=k_1+\cdots+k_r,\quad r\geq1.
 \]
 For $d+1=K_r$ and $b_{d+1}=1$ for some $r\geq 1$, one has 
 \[
 b_{d+1}2^{d+1} \left(    \sum_{m_1\ge\cdots\ge m_d\ge3}
 \frac{b_2^{m_1-m_2}\cdots b_{d}^{m_{d-1}-m_d}}
 {m_1^2m_2\cdots m_d}       \right)= 2^{K_r}\sum_{n_1\geq \cdots \geq n_r\geq 3}\frac{1}{n_1^{k_1} n_2^{k_2}\cdots n_r^{k_r}    }.
   \]
   Here the above formula holds because only the following terms
   \[
   m_1=\cdots=m_{K_1-1}
   \]
   \[
    m_{K_1}=\cdots =m_{K_2-1}
   \]
   \[
   \cdots
   \]
   \[
   m_{K_{r-1}}=\cdots =m_{K_r-1}
   \]
   on the left hand side give non-zero contributions.
 Therefore
 \begin{equation}\label{eq:su3}
 \begin{split}
 &\;\;\;\;\Phi^\prime(y)\\
 & =1+\sum_{r=1}^{+\infty} 2^{K_r} \left(   \sum_{n_1\geq \cdots \geq n_r\geq 3}\frac{1}{n_1^{k_1} n_2^{k_2}\cdots n_r^{k_r}    }
     \right)\\
 &=1+\sum_{r=1}^{+\infty} \sum_{n_1\geq\cdots \geq n_r\geq 3} \left(  \frac{2}{n_1}  \right)^{k_1}  \left(  \frac{2}{n_2}  \right)^{k_2}   \cdots  \left(  \frac{2}{n_r}  \right)^{k_r} .     \\
     \end{split}
 \end{equation}
 On the other hand, 
 \begin{equation}\label{eq:r2}
 \begin{split}
 &\;\;\;\;\sum_{n_1\geq \cdots \geq n_r\geq 2}      \left(\frac{2}{n_1}\right)^{k_1}\cdots  \left(\frac{2}{n_r}\right)^{k_r} \\
 &=  \left(    \sum_{\substack{n_1\geq \cdots \geq n_r\geq 2\\ n_1=\cdots=n_r=2   }} +   \sum_{\substack{n_1\geq \cdots \geq n_r\geq 2\\ n_1\geq 3, n_2=\cdots=n_r=2   }}+\cdots+   \sum_{\substack{n_1\geq \cdots \geq n_r\geq 2\\ n_1\geq \cdots\geq n_r\geq 3   }}      \right)    \left(\frac{2}{n_1}\right)^{k_1}\cdots  \left(\frac{2}{n_r}\right)^{k_r}   \\
 &=1+\sum_{s=1}^r   \sum_{n_1\geq \cdots \geq n_s\geq 3}    \left(\frac{2}{n_1}\right)^{k_1}\cdots  \left(\frac{2}{n_s}\right)^{k_s} . \end{split}\end{equation}
   Thus by \eqref{eq:su3} and \eqref{eq:r2}, the formula \eqref{der} is proved. By the same analysis, one can also prove \eqref{eq:fi}.
  \end{proof}

\subsection{Finite dyadic coordinate}

For an admissible index $\mathbf{k}=(k_1,\ldots,k_r)$, define
$K_j=k_1+\cdots+k_j$ and 
\begin{equation}\label{eq:beta-finite}
 \beta^f(\mathbf{k})
 =\sum_{j=1}^{r}2^{-K_j}.
\end{equation}
Let $\Dset_r$ be the set of dyadic rational numbers in $\left(0,\frac{1}{2}\right)$ whose terminating binary
expansion contains exactly $r$ digits equal to $1$:
\begin{equation}\label{eq:Dr}
 \Dset_r
 :=\left\{\sum_{j=1}^{r}2^{-n_j}:
 2\le n_1<\cdots<n_r\right\}.
\end{equation}
For $y\in\Dset_r$, let $\nu(y)$ be the position of the last binary digit equal
to $1$.

\begin{proposition}\label{prop:finite-coding}
The map $\mathbf{k}\mapsto\beta^f(\mathbf{k})$ is a bijection from $\Sset_r$ to
$\Dset_r$.  Moreover,
\begin{equation}\label{eq:finite-code-properties}
 \nu\bigl(\beta^f(\mathbf{k})\bigr)=\wt(\mathbf{k}),
 \qquad
 \Phi\bigl(\beta^f(\mathbf{k})\bigr)=\zeta^\star(\mathbf{k}).
\end{equation}
\end{proposition}

\begin{proof}
The exponents $K_j$ in \eqref{eq:beta-finite} are strictly increasing, so
$\beta^f(\mathbf{k})\in\Dset_r$ and its last-one position is
$K_r=\wt(\mathbf{k})$.  Conversely, from
$2\le n_1<\cdots<n_r$ recover
\[
 k_1=n_1,
 \qquad
 k_j=n_j-n_{j-1}\quad(2\le j\le r).
\]
As 
\[
\beta^f(k_1,\cdots,k_r)=\beta(k_1,\cdots,k_{r-1},k_r+1,\{1\}^{\infty})
\]
and 
\[
\mathop{\mathrm{lim}}_{s\rightarrow+\infty} \zeta^{\star} (k_1,\cdots,k_{r-1},k_r+1,\{1\}^{s})=\zeta^\star(k_1,\cdots,k_{r-1},k_r),
\]
the formula $    \Phi\bigl(\beta^f(\mathbf{k})\bigr)=\zeta^\star(\mathbf{k})   $ follows immediately.
\end{proof}

The identity in \cref{prop:finite-coding} converts the normalized error into
\begin{equation}\label{eq:Delta-Phi-common}
 \Delta_r(\Phi(x))
 =\inf_{y\in\Dset_r}
 2^{\nu(y)}\,|\Phi(x)-\Phi(y)|,
 \qquad 0<x<\frac{1}{2}.
\end{equation}

\subsection{Multiple zeta-star values}

For an admissible finite index
\[
 \mathbf k=(k_1,\ldots,k_r),
 \qquad k_1\ge2,\quad k_2,\ldots,k_r\ge1,
\]
the associated multiple zeta-star value is
\[
 \zeta^\star(\mathbf k)
 =\sum_{n_1\ge\cdots\ge n_r\ge1}
 \frac{1}{n_1^{k_1}\cdots n_r^{k_r}}.
\]
Its weight and depth are, respectively,
\[
 \wt(\mathbf k)=k_1+\cdots+k_r,
 \qquad
 \dep(\mathbf k)=r.
\]
We write $\mathcal{Z}^{\star}$ for the set of all multiple zeta-star values. Denote by $\mathcal{S}$ the set of admissible finite indices.
Define 
\[
(k_1,\cdots,k_r,k_{r+1})\succ(k_1,\cdots,k_r)\]
and 
\[
(k_1,\cdots,k_r)\succ(l_1,\cdots,l_s)\]
for $$(k_1,\cdots,k_{i-1})=(l_1,\cdots,l_{i-1}),\quad k_i<l_i $$ and $1\leq i\leq \mathrm{min}\{r,s\}$. Then $\mathcal{S}$ is a totally ordered set. By the same analysis one can also define a total order on the set $\mathcal{T}$.

We use the following results from the  theory of zeta-star  correspondence
\cite{LiTopology}.

\begin{theorem}\label{thm:Li-input}
The following facts hold.
\begin{enumerate}[label=(\alph*),leftmargin=2.2em]
\item  The total order of $\mathcal{S}$ is compatible with the order of multiple zeta-star values:
\[
  (k_1,\cdots, k_r)\succ (l_1,\cdots,l_s)  \Leftrightarrow  \zeta^\star(k_1,\cdots,k_r)>\zeta^\star(l_1,\cdots,l_s).  \]
\item Let
\[
 \Tset=\bigl\{(k_1,k_2,\ldots): k_1\ge2,\ k_i\ge1,
 \text{ and if }k_1=2\text{ then some }k_s\ge2\ (s\ge2)\bigr\}.
\]
For $\mathbf k\in\Tset$, the limit
\[
 \eta(\mathbf k)=\lim_{r\to\infty}\zeta^\star(k_1,\ldots,k_r)
\]
exists, and $\eta:\Tset\to(1,\infty)$ is an order-preserving bijection.
\item As a result, $\mathcal{Z}^{\star}$ is countable and dense in $(1,\infty)$.
\end{enumerate}
\end{theorem}

\subsection{Basic intervals}

For $\mathbf{k}=(k_1,\ldots,k_r)$, let
\[
 j=\max\{i:k_i\ge2\}.
\]
Define the upper endpoint by
\begin{equation}\label{eq:U}
 U(\mathbf{k})=
 \begin{cases}
 +\infty,&\mathbf{k}=(2,\ones{r-1}),\\[1mm]
 \zeta^\star(k_1,\ldots,k_{j-1},k_j-1),&\text{otherwise}.
 \end{cases}
\end{equation}
The corresponding basic interval is
\[
 Z_{\mathbf{k}}=(\zeta^\star(\mathbf{k}),U(\mathbf{k})].
\]
At each fixed depth these intervals are pairwise disjoint and cover
$(1,+\infty)$.  By Lemma $4.4$ in \cite{LiDA} (or by the proof of \Cref{lem:natural-errors}  in \Cref{sec:topology}), their lengths are
\begin{equation}\label{eq:interval-length}
 \ell(\mathbf{k})
 :=U(\mathbf{k})-\zeta^\star(\mathbf{k})
 =\sum_{n_1\ge\cdots\ge n_r\ge2}
 \frac{n_r-1}{n_1^{k_1}\cdots n_r^{k_r}},
\end{equation}
with
\begin{equation}\label{eq:infinite-interval}
 \ell(\mathbf{k})=+\infty
 \quad\Leftrightarrow\quad
 \mathbf{k}=(2,\ones{r-1}).
\end{equation}
We shall frequently use the normalized interval length
\begin{equation}\label{eq:Lnorm}
 \Lnorm(\mathbf{k}):=2^{\wt(\mathbf{k})}\ell(\mathbf{k}).
\end{equation}
By \eqref{eq:interval-length},
\begin{equation}\label{eq:Lnorm-series}
 \Lnorm(\mathbf{k})
 =\sum_{n_1\ge\cdots\ge n_r\ge2}
 (n_r-1)\prod_{i=1}^{r}\left(\frac{2}{n_i}\right)^{k_i}.
\end{equation}
Thus $\Lnorm(\mathbf{k})$ is coordinatewise nonincreasing in the entries of
$\mathbf{k}$.

\section{Topological structure and evaluations at multiple zeta-star values}\label{sec:topology}

\subsection{Upper semicontinuity and a residual zero set}

\begin{proposition}\label{prop:regularity}
For every fixed $r$, the map $\alpha\mapsto\Delta_r(\alpha)$ is upper
semicontinuous.  The map $\alpha\mapsto\mathcal{N}(\alpha)$ is Borel
measurable as an extended real-valued function.
\end{proposition}

\begin{proof}
For each $\mathbf{k}\in\Sset_r$, the map
\[
 \alpha\longmapsto
 2^{\wt(\mathbf{k})}\bigl|\alpha-\zeta^\star(\mathbf{k})\bigr|
\]
is continuous.  Hence, for every $c\in\R$,
\[
 \{\alpha:\Delta_r(\alpha)<c\}
 =\bigcup_{\mathbf{k}\in\Sset_r}
 \left\{\alpha:
 2^{\wt(\mathbf{k})}\bigl|\alpha-\zeta^\star(\mathbf{k})\bigr|<c
 \right\}
\]
is open.  This is equivalent to upper semicontinuity.  Finally,
\[
 \mathcal{N}(\alpha)
 =\sup_{N\ge1}\inf_{r\ge N}\Delta_r(\alpha),
\]
so $\alpha\mapsto\mathcal{N}(\alpha)$ is obtained from Borel functions by
countable infima and suprema.
\end{proof}

\begin{theorem}\label{thm:residual-zero}
The set
\[
 \mathcal{Z}_0:=\{\alpha>1:\mathcal{N}(\alpha)=0\}
\]
is a dense $G_\delta$ subset of $(1,+\infty)$.  Its intersection with every
nonempty open interval is uncountable.
\end{theorem}

\begin{proof}
For $m,N\ge1$, put
\[
 \mathcal{U}_{m,N}
 :=\bigcup_{r\ge N}\{\alpha:\Delta_r(\alpha)<1/m\}.
\]
By \cref{prop:regularity}, each $\mathcal{U}_{m,N}$ is open.  It is dense:
if $J\subset(1,+\infty)$ is a nonempty open interval, choose
$\alpha=\eta(k_1,k_2,\ldots)\in J$.  The prefix values
\[
 \alpha_r:=\zeta^\star(k_1,\ldots,k_r)
\]
converge to $\alpha$, so some $\alpha_r\in J$ with $r\ge N$.  At that point
$\Delta_r(\alpha_r)=0$, hence $J\cap\mathcal{U}_{m,N}\ne\varnothing$.
Because all $\Delta_r(\alpha)$ are nonnegative,
\[
 \mathcal{Z}_0
 =\bigcap_{m=1}^{\infty}\bigcap_{N=1}^{\infty}\mathcal{U}_{m,N}.
\]
The Baire category theorem proves density and the $G_\delta$ property.  As a countable set is of the first category, a
residual subset of an interval cannot be countable.
\end{proof}

\subsection{Natural lower prefixes}

Fix a finite admissible index
$\mathbf{k}=(k_1,\ldots,k_s)$ and put
\[
 K=\wt(\mathbf{k}),
 \qquad
 \alpha=\zeta^\star(\mathbf{k}).
\]
The infinite sequence corresponding to $\alpha$ is
\[
 \mathbf{p}=\eta^{-1}(\alpha)=(k_1,\ldots,k_{s-1},k_s+1,\{1\}^{\infty}).
\]
For $d\ge0$, define
\begin{align}
 q_d&:=\zeta^\star(k_1,\ldots,k_{s-1},k_s+1,\ones{d}),\label{eq:qd}\\
 D_d&:=\alpha-q_d,
 \qquad
 B_d:=2^{K+d+1}D_d.\label{eq:Bd}
\end{align}
The value $q_d$ has depth $s+d$, weight $K+d+1$, and $q_d\rightarrow\alpha$.

\begin{lemma}\label{lem:natural-errors}
The sequence $B_d$ is strictly increasing, satisfies $B_{d+1}<2B_d$, and
\begin{equation}\label{eq:Bd-limit}
 \lim_{d\to+\infty}B_d=\Lnorm(\mathbf{k}),
\end{equation}
where the limit may be $+\infty$.
\end{lemma}

\begin{proof}
By using the simple observation 
\[
\frac{1}{n}\sum_{n\geq m\geq1}1=1
\]
repeatedly, 
we have 
\begin{equation}\label{eq:Dd-ser}
\begin{split}
 &\;\;\;\;D_d\\
   &= \sum_{n_1\geq \cdots \geq n_s\geq 1}    \frac{n_s}{ n_1^{k_1}\cdots n_{s-1}^{k_{s-1}}
  n_s^{k_s+1} }   -       \sum_{n_1\ge\cdots\ge n_{s+d}\ge 1}
 \frac{1}{n_1^{k_1}\cdots n_{s-1}^{k_{s-1}}
  n_s^{k_s+1}n_{s+1}\cdots n_{s+d}}        \\
    &= \sum_{n_1\geq \cdots \geq n_{s+1}\geq 1}    \frac{n_{s+1}}{ n_1^{k_1}\cdots n_{s-1}^{k_{s-1}}
  n_s^{k_s+1} n_{s+1}}   -       \sum_{n_1\ge\cdots\ge n_{s+d}\ge 1}
 \frac{1}{n_1^{k_1}\cdots n_{s-1}^{k_{s-1}}
  n_s^{k_s+1}n_{s+1}\cdots n_{s+d}}        \\   
  &\qquad \cdots   \qquad   \\
  &= \sum_{n_1\ge\cdots\ge n_{s+d}\ge1}
 \frac{n_{s+d}-1}{n_1^{k_1}\cdots n_{s-1}^{k_{s-1}}
  n_s^{k_s+1}n_{s+1}\cdots n_{s+d}}\\  
   &= \sum_{n_1\ge\cdots\ge n_{s+d}\ge2}
 \frac{n_{s+d}-1}{n_1^{k_1}\cdots n_{s-1}^{k_{s-1}}
  n_s^{k_s+1}n_{s+1}\cdots n_{s+d}}\\
  \end{split}
\end{equation}
For $N\ge2$, define
\[
 R_0(N)=N-1,
 \qquad
 R_{d+1}(N)=\sum_{m=2}^{N}\frac{R_d(m)}{m},
 \qquad
 h_d(N)=\frac{2^{d+1}R_d(N)}{N}.
\]
Then
\begin{equation}\label{eq:h-recurrence-consolidated}
 h_{d+1}(N)=\frac{2}{N}\sum_{m=2}^{N}h_d(m).
\end{equation}
The recurrence shows inductively that
\[
 h_d(N)\uparrow N-1.
\]
To justify both assertions, let
\[
 (Tf)(N):=\frac{2}{N}\sum_{m=2}^{N}f(m),
\]
so that $h_{d+1}=Th_d$.  Since
\[
 h_1(N)=\frac{4}{N}\sum_{m=2}^{N}\left(1-\frac1m\right)
 \ge \frac{2(N-1)}{N}=h_0(N),
\]
where we used $\sum_{m=2}^{N}m^{-1}\le(N-1)/2$. Since $T$ preserves
pointwise order, the sequence 
\[
h_d(N), \quad d\geq 0
\]
is nondecreasing for $N\geq 2$.  Because $h_0(N)\le N-1$ and the function $N\mapsto N-1$ is a fixed
point of  the operator $T$, one has 
\[
h_d(N)\leq N-1
\]
as the operator $T$ preserves pointwise order.
Let \[L_N=\lim_{d\rightarrow +\infty} h_d(N).\] Then $L_2=1$ and, for $N\ge3$,
\[
 L_N=\frac{2}{N}\sum_{m=2}^{N}L_m.
\]
Induction on $N$ gives $L_N=N-1$.  The inequality $h_1(N)>h_0(N)$ for
$N\ge3$, followed again by order preservation, also gives strict increase at
all $N\ge3$.

Grouping \eqref{eq:Dd-ser} by $n_s=N$ gives
\[
 B_d
 =2^K\sum_{n_1\ge\cdots\ge n_s\ge2}
 \frac{h_d(n_s)}{n_1^{k_1}\cdots n_s^{k_s}}.
\]
Monotone convergence and \eqref{eq:interval-length} imply
\eqref{eq:Bd-limit}.  Strict increase follows because the outer sum contains
terms with $n_s\ge3$.  Finally, $D_{d+1}<D_d$ gives
$B_{d+1}<2B_d$.
\end{proof}

\section{Metric approximation and almost-everywhere vanishing}
\label{sec:metric}

Throughout this section, $m$ denotes Lebesgue measure.  Because
$(1,+\infty)$ has infinite measure, ``full measure'' means that the complement
inside $(1,+\infty)$ has measure zero.

\subsection{Tail radii and the convergence condition}

For $g:\NN\to[1,+\infty)$ and an admissible index
$\mathbf{k}=(k_1,\ldots,k_r)$, define \begin{equation}\label{eq:tail-radius}
 R_g(\mathbf{k})
 :=\sum_{n_1\ge\cdots\ge n_r\ge n_{r+1}\ge2}
 \frac{1}{n_1^{k_1}\cdots n_r^{k_r}n_{r+1}^{g(r)}}.
\end{equation}
Let $\mathcal{F}_g$ be the set of $\alpha>1$ for which
\[
 \bigl|\alpha-\zeta^\star(\mathbf{k})\bigr|<R_g(\mathbf{k})
\]
holds for infinitely many indices whose depths tend to infinity.  The
zero--one law from \cite{LiDA} states that
\begin{equation}\label{eq:tail-zero-one}
 \sum_{r=1}^{\infty}2^{-g(r)}<+\infty
 \Longrightarrow m(\mathcal{F}_g)=0,
 \qquad
 \sum_{r=1}^{\infty}2^{-g(r)}=+\infty
 \Longrightarrow m((1,+\infty)\setminus\mathcal{F}_g)=0.
\end{equation}
The direct-radius set $\mathcal{B}_g$ was defined in \eqref{eq:Bg-intro}.
The single term $n_1=\cdots=n_{r+1}=2$ in \eqref{eq:tail-radius} equals
$2^{-\wt(\mathbf{k})-g(r)}$, and all remaining terms are positive.  Hence
\begin{equation}\label{eq:B-subset-F}
 \mathcal{B}_g\subseteq\mathcal{F}_g.
\end{equation}

\begin{theorem}\label{thm:metric-conv}
If $\sum_{r\ge1}2^{-g(r)}<+\infty$, then $m(\mathcal{B}_g)=0$.
\end{theorem}

\begin{proof}
Combine \eqref{eq:B-subset-F} with the convergence implication in
\eqref{eq:tail-zero-one}.
\end{proof}

\begin{corollary}\label{cor:eventual-lower}
For every $\varepsilon>0$, for Lebesgue almost every $\alpha>1$ one has
\[
 \Delta_r(\alpha)\ge r^{-(1+\varepsilon)}
\]
for all sufficiently large $r$.
\end{corollary}

\begin{proof}
Take $g(r)=\max\{1,(1+\varepsilon)\log_2r\}$.  Then
$\sum_r2^{-g(r)}<+\infty$, and \cref{thm:metric-conv} applies.
\end{proof}

\subsection{Large next digits and normalized cylinder lengths}

For $\alpha>1$, write
\[
 \eta^{-1}(\alpha)=(k_1(\alpha),k_2(\alpha),\ldots).
\]
We use the following recurrence result (Lemma $4.6$ in \cite{LiDA}).

\begin{lemma}\label{lem:large-digit}
Let $H:\NN\to[1,+\infty)$.  If
\[
 \sum_{r=1}^{\infty}2^{-H(r)}=+\infty,
\]
then, for Lebesgue almost every $\alpha>1$,
\begin{equation}\label{eq:large-digit}
 k_{r+1}(\alpha)\ge\lfloor H(r)\rfloor+2
\end{equation}
for infinitely many $r$.
\end{lemma}

The next estimate controls the normalized lengths of the natural cylinders.

\begin{lemma}\label{lem:linear-cylinder}
Let ${\bf k}=(k_1,k_2,\ldots)\in\Tset$.  There exist
$r_0({\bf k})$ and $C({\bf k})>0$ such that
\begin{equation}\label{eq:linear-cylinder}
 \Lnorm(k_1,\ldots,k_r)
 \le C({\bf k})r
 \qquad(r\ge r_0({\bf k})).
\end{equation}
More explicitly, let $\mathbf{b}=(2,\{1\}^{\infty})$ and
\[
 t=\min\{i\ge1:k_i>b_i\}.
\]
Then one may take $r_0=t$ and
\begin{equation}\label{eq:explicit-cylinder-constant}
 C({\bf k})=2^{t+1}(1+2\log2).
\end{equation}
\end{lemma}

\begin{proof}
The integer $t$ exists because ${\bf k}\in\Tset$ excludes the
minimal sequence $(2,\{1\}^{\infty})$.  For $d\ge0$, define
\[
 \mathbf{q}_{1,d}=(3,\ones{d})
\]
and, for $t\ge2$,
\[
 \mathbf{q}_{t,d}=(2,\ones{t-2},2,\ones{d}).
\]
If $r=t+d$, then $(k_1,\ldots,k_r)$ dominates $\mathbf{q}_{t,d}$
coordinatewise.  The monotonicity following \eqref{eq:Lnorm-series} gives
\[
 \Lnorm(k_1,\ldots,k_r)\le\Lnorm(\mathbf{q}_{t,d}).
\]

Use the functions $R_d(N)$ and $h_d(N)$ from the proof of
\cref{lem:natural-errors}.  Besides $h_d(N)\le N-1$, the recurrence
\eqref{eq:h-recurrence-consolidated} gives
\begin{equation}\label{eq:h-two-bounds}
 h_d(N)\le\min\{N-1,2^{d+1}\}.
\end{equation}
Indeed, $h_0(N)<2$, and the implication
$h_d\le2^{d+1}\Rightarrow h_{d+1}\le2^{d+2}$ follows immediately from the
recurrence.

For $t=1$, grouping the interval series by the first summation variable gives
\begin{equation}\label{eq:q1-cylinder}
 \Lnorm(\mathbf{q}_{1,d})
 =4\sum_{N=2}^{\infty}\frac{h_d(N)}{N^2}.
\end{equation}
For $t\ge2$, define
\[
 E_t(N):=\frac1N
 \sum_{n_1\ge\cdots\ge n_{t-1}\ge N}
 \frac{1}{n_1^2n_2\cdots n_{t-1}}.
\]
Grouping instead by the variable in the second exponent-$2$ position gives
\begin{equation}\label{eq:qt-cylinder}
 \Lnorm(\mathbf{q}_{t,d})
 =2^{t+1}\sum_{N=2}^{\infty}E_t(N)h_d(N).
\end{equation}
We claim that
\begin{equation}\label{eq:Et-bound}
 E_t(N)\le\frac{1}{N(N-1)}.
\end{equation}
Define
\[
 F_j(N):=\sum_{n_1\ge\cdots\ge n_j\ge N}
 \frac{1}{n_1^2n_2\cdots n_j}.
\]
It is clear that 
\[
F_1(N)=\sum_{n_1\geq N}\frac{1}{n_1^2}< \sum_{n_1\geq N}\frac{1}{n_1(n_1-1)}=\frac{1}{N-1}
\]
and 
\[
 F_j(N)=\sum_{m=N}^{\infty}\frac{F_{j-1}(m)}{m}.
\]
By induction, from $F_{j-1}(m)\leq \frac{1}{m-1} $ it follows immediately that 
\[
F_j(N)\leq \frac{1}{N-1}.
\]
Since $E_t(N)=F_{t-1}(N)/N$,
\eqref{eq:Et-bound} follows.

Set $M=2^{d+1}$.  From \eqref{eq:h-two-bounds} and
\eqref{eq:q1-cylinder},
\begin{align*}
 \Lnorm(\mathbf{q}_{1,d})
 &\le4\left(
 \sum_{N=2}^{M+1}\frac1N
 +M\sum_{N=M+2}^{\infty}\frac1{N^2}\right)\\
 &\le4\bigl(\log(M+1)+1\bigr).
\end{align*}
Likewise, \eqref{eq:qt-cylinder} and \eqref{eq:Et-bound} yield, for $t\ge2$,
\begin{align*}
 \Lnorm(\mathbf{q}_{t,d})
 &\le2^{t+1}\left(
 \sum_{N=2}^{M+1}\frac1N
 +M\sum_{N=M+2}^{\infty}\frac1{N(N-1)}\right)\\
 &\le2^{t+1}\bigl(\log(M+1)+1\bigr).
\end{align*}
Because $M=2^{d+1}$,
\[
 \log(M+1)+1\le(1+2\log2)(d+1).
\]
Finally, $d+1\le t+d=r$, which proves
\eqref{eq:linear-cylinder} with the constant in
\eqref{eq:explicit-cylinder-constant}.
\end{proof}

\begin{lemma}\label{lem:prefix-estimate}
Let
\[
 \alpha=\eta(k_1,k_2,\ldots),
 \qquad
 K_r=k_1+\cdots+k_r,
 \qquad
 x_r=\zeta^\star(k_1,\ldots,k_r).
\]
There are $C(\alpha)>0$ and $r_0(\alpha)$ such that, whenever
$r\ge r_0(\alpha)$ and $k_{r+1}\ge2$,
\begin{equation}\label{eq:prefix-estimate}
 2^{K_r}|\alpha-x_r|
 \le C(\alpha)r\,2^{1-k_{r+1}}.
\end{equation}
Consequently,
\begin{equation}\label{eq:Delta-prefix}
 \Delta_r(\alpha)
 \le C(\alpha)r\,2^{1-k_{r+1}}.
\end{equation}
\end{lemma}

\begin{proof}
Put $j=k_{r+1}\ge2$.  The order structure gives
\[
 0\le\alpha-x_r
 \le\zeta^\star(k_1,\ldots,k_r,j-1)-x_r.
\]
As
\[
\begin{split}
&\;\;\;\;  \zeta^\star(k_1,\ldots,k_r,j-1)-x_r      \\
&=  \sum_{n_1\geq \cdots \geq n_r\geq n_{r+1}\geq 2  }   \frac{1}{n_1^{k_1}\cdots n_r^{k_r} n_{r+1}^{j-1}}          \\
&\leq   \sum_{n_1\geq \cdots \geq n_r\geq n_{r+1}\geq 2  }   \frac{1}{n_1^{k_1}\cdots n_r^{k_r} 2^{j-1}}          \\
&=2^{1-j} \sum_{n_1\geq \cdots \geq n_r\geq 2  }   \frac{n_r-1}{n_1^{k_1}\cdots n_r^{k_r} }     \end{split}
\]
  Hence
\[
 \alpha-x_r
 \le2^{1-j}\ell(k_1,\ldots,k_r).
\]
Multiplication by $2^{K_r}$ and \cref{lem:linear-cylinder} yield
\eqref{eq:prefix-estimate}.  Since $x_r$ is an admissible depth $r$
approximant, \eqref{eq:Delta-prefix} follows.
\end{proof}

\subsection{A divergence theorem}

The divergence conjecture proposed in \cite{LiDA} asks
whether
\[
 \sum_{r=1}^{\infty}2^{-g(r)}=+\infty
 \quad\Longrightarrow\quad
 m((1,+\infty)\setminus\mathcal{B}_g)=0.
\]
The following theorem proves a logarithmically reinforced divergence case.

\begin{theorem}\label{thm:re-div}
If
\[
 \sum_{r=4}^{\infty}\frac{2^{-g(r)}}{r\log_2r}=+\infty,
\]
then
\[
 m((1,+\infty)\setminus\mathcal{B}_g)=0.
\]
\end{theorem}

\begin{proof}
For $r\ge4$, define
\[
 H(r)=g(r)+\log_2r+\log_2\log_2r,
\]
and choose the first three values of $H$ arbitrarily in $[1,+\infty)$.  Then
\[
 \sum_{r=4}^{\infty}2^{-H(r)}
 =\sum_{r=4}^{\infty}\frac{2^{-g(r)}}{r\log_2r}=+\infty.
\]
By \cref{lem:large-digit}, for almost every $\alpha$ there are infinitely many
$r$ for which
$k_{r+1}(\alpha)\ge\lfloor H(r)\rfloor+2$.  Along those depths,
\[
 2^{1-k_{r+1}(\alpha)}
 \le2^{-\lfloor H(r)\rfloor-1}
 \le2^{-H(r)}.
\]
For all sufficiently large selected depths, \cref{lem:prefix-estimate} gives
\[
 \Delta_r(\alpha)
 \le C(\alpha)r2^{-H(r)}
 =\frac{C(\alpha)}{\log_2r}\,2^{-g(r)}
 <2^{-g(r)}.
\]
Thus $\alpha\in\mathcal{B}_g$ for almost every $\alpha>1$.
\end{proof}

\begin{theorem}\label{thm:ae-zero}
One has
\[
 \mathcal{N}(\alpha)=0
 \qquad\text{for Lebesgue almost every }\alpha>1.
\]
\end{theorem}

\begin{proof}
For each integer $m\ge1$, take $g_m(r)\equiv m$.  Since
\[
 \sum_{r=4}^{\infty}\frac{2^{-m}}{r\log_2r}=\infty,
\]
\cref{thm:re-div} shows that every $\mathcal{B}_{g_m}$ has full
measure.  Moreover,
\[
 \{\alpha>1:\mathcal{N}(\alpha)=0\}
 =\bigcap_{m=1}^{\infty}\mathcal{B}_{g_m},
\]
because a nonnegative liminf is zero exactly when every dyadic threshold
$2^{-m}$ is crossed infinitely often.  A countable intersection of
full-measure sets has full measure.
\end{proof}

\subsection{An explicit uncountable family of zeros}

The metric theorem is complemented by a direct symbolic construction.
Lemma~4.3 of \cite{LiTopology} gives a uniform tail estimate: there is an
absolute
constant $C>0$ such that
\begin{equation}\label{eq:uniform-tail}
 \sum_{n_1\ge\cdots\ge n_r\ge n_{r+1}\ge2}
 \frac{1}{n_1^{k_1}\cdots n_r^{k_r}n_{r+1}}
 \le C\,2^{-(k_1+\cdots+k_r)}
\end{equation}
whenever $k_1,\ldots,k_r\ge2$.

\begin{proposition}\label{prop:explicit-zeros}
If
\[
 \alpha=\eta(k_1,k_2,\ldots),
 \qquad k_i\ge2\text{ for all }i,
 \qquad \sup_i k_i=+\infty,
\]
then $\mathcal{N}(\alpha)=0$.
\end{proposition}

\begin{proof}
Let $K_r=k_1+\cdots+k_r$ and
$x_r=\zeta^\star(k_1,\ldots,k_r)$.  If $m=k_{r+1}\ge2$, then
\[
 0<\alpha-x_r
 <\zeta^\star(k_1,\ldots,k_r,m-1)-x_r.
\]
For $n\ge2$, $n^{-(m-1)}\le2^{2-m}n^{-1}$.  Applying
\eqref{eq:uniform-tail} gives
\[
 2^{K_r}(\alpha-x_r)\le C2^{2-k_{r+1}}.
\]
Along a subsequence for which $k_{r+1}\to+\infty$, the right-hand side tends
to zero.  Since $x_r$ is a depth $r$ approximant,
$0\le\Delta_r(\alpha)\le2^{K_r}(\alpha-x_r)$ along that subsequence, and
therefore $\mathcal{N}(\alpha)=0$.
\end{proof}

There are uncountably many unbounded sequences with all digits at least $2$,
so \cref{prop:explicit-zeros} gives an explicit uncountable subfamily of the
zero set found in \cref{thm:residual-zero,thm:ae-zero}.

\section{Evaluations of the normalized approximation function}\label{sec:integers}

\subsection{The binary liminf constant}

For $x\in\left(0,\frac{1}{2}\right)$, define
\begin{equation}\label{eq:binary-delta}
 \delta_r(x):=\inf_{y\in\Dset_r}2^{\nu(y)}|x-y|,
 \qquad
 \Lambda(x):=\liminf_{r\to+\infty}\delta_r(x).
\end{equation}
For $u\in\R$, denote by $\|u\|=\dist(u,\mathbb Z)$.

\begin{lemma}\label{lem:orbit-formula}
If $x\in\left(0,\frac{1}{2}\right)$ is non-dyadic, then
\begin{equation}\label{eq:orbit-formula}
 \Lambda(x)=\liminf_{n\to\infty}\|2^nx\|.
\end{equation}
\end{lemma}

\begin{proof}
Let $\beta_0=\liminf_{n\to+\infty}\|2^nx\|$.  If
$y=m/2^{\nu(y)}\in\Dset_r$ is in reduced form, then $m$ is odd and
$\nu(y)\ge r+1$, so
\[
 2^{\nu(y)}|x-y|
 =|2^{\nu(y)}x-m|
 \ge\|2^{\nu(y)}x\|
 \ge\inf_{n\ge r+1}\|2^nx\|.
\]
Taking the infimum and then the lower limit gives
$\Lambda(x)\ge \beta_0$.

Conversely, choose $n_j\to+\infty$ with
$\|2^{n_j}x\|\to\beta_0$, let $m_j$ be a nearest integer to $2^{n_j}x$, and
reduce $m_j/2^{n_j}$ to a terminating dyadic $y_j$.  If $r_j$ is the number
of ones in its terminating expansion, then
\[
 2^{\nu(y_j)}|x-y_j|
 \le2^{n_j}|x-y_j|
 =\|2^{n_j}x\|.
\]
Because $y_j\to x$ and $x$ is non-dyadic, the numbers $r_j$ must tend to
infinity.  Indeed, otherwise there would be an integer $R$ and a subsequence
with $r_j\le R$.  After padding with zero summands,  each element of that
subsequence can be written as
\[
 y_j=\sum_{\ell=1}^{R}2^{-n_{\ell,j}},
 \qquad n_{\ell,j}\in\NN\cup\{+\infty\},
 \qquad 2^{-\infty}:=0,
\]
with distinct finite exponents.  By successive subsequence extraction, each
$n_{\ell,j}$ is eventually constant or tends to $+\infty$.  The subsequential
limit is therefore a finite sum of negative powers of $2$, hence dyadic.  This
contradicts $y_j\to x$.  Taking lower limits along $r_j$ gives
$\Lambda(x)\le\beta_0$.
\end{proof}

\begin{corollary}\label{cor:inva}
Let $L\ge1$, let $A\in\{0,1,\ldots,2^L-1\}$, and suppose
\[
 P(x)=\frac{A+x}{2^{L}}\in\left(0,\frac{1}{2}\right).
\]
If $x$ is non-dyadic, then
\begin{equation}\label{eq:prefix-invariance}
 \Lambda(P(x))=\Lambda(x).
\end{equation}
\end{corollary}

\begin{proof}
For every $n\ge L$, the fractional parts of $2^nP(x)$ and $2^{n-L}x$ agree.
Apply \cref{lem:orbit-formula}.
\end{proof}

\begin{proposition}\label{prop:differentiable}
Let $x\in\left(0,\frac{1}{2}\right)$ be non-dyadic.  If $\Phi$ is differentiable at $x$ and
$$0<\Phi'(x)<+\infty,$$ then
\begin{equation}\label{eq:differentiable-transfer}
 \mathcal{N}(\Phi(x))=\Phi^\prime(x)\Lambda(x).
\end{equation}
\end{proposition}

\begin{proof}
Choose depths $r_j\to+\infty$ such that
$\delta_{r_j}(x)\to\Lambda(x)$, and choose $y_j\in\Dset_{r_j}$ with
\[
 2^{\nu(y_j)}|x-y_j|\le\delta_{r_j}(x)+\frac1j.
\]
The right-hand side is bounded and $\nu(y_j)\ge r_j+1$, so $y_j\to x$.  Hence
\[
 \frac{|\Phi(x)-\Phi(y_j)|}{|x-y_j|}\longrightarrow\Phi'(x).
\]
Using \eqref{eq:Delta-Phi-common} along this subsequence gives
\[
 \mathcal{N}(\Phi(x))\le\Phi'(x)\Lambda(x).
\]
In particular, the left-hand side is finite.

For the reverse inequality, choose $r_j\to+\infty$ with
$\Delta_{r_j}(\Phi(x))\to\mathcal{N}(\Phi(x))$, and choose
$y_j\in\Dset_{r_j}$ such that
\[
 2^{\nu(y_j)}\Big{|}\Phi(x)-\Phi(y_j)\Big{|}
 \le\Delta_{r_j}(\Phi(x))+\frac1j.
\]
These normalized quantities are bounded.  Since $\nu(y_j)\ge r_j+1$, it follows
that $\Phi(y_j)\to\Phi(x)$; continuity of the inverse of the strictly
increasing bijection $\Phi$ then gives $y_j\to x$.  Therefore
\[
 \Delta_{r_j}(\Phi(x))+\frac1j
 \ge
 \frac{|\Phi(x)-\Phi(y_j)|}{|x-y_j|}\,
 2^{\nu(y_j)}|x-y_j|
 \ge
 \frac{|\Phi(x)-\Phi(y_j)|}{|x-y_j|}\,\delta_{r_j}(x).
\]
Taking lower limits proves
$\mathcal{N}(\Phi(x))\ge\Phi'(x)\Lambda(x)$.
\end{proof}

 \begin{lemma}\label{lem:conv}
 For $r\geq 2, k_1\geq 2, k_2,\cdots,k_r\geq 1$ and 
 \[
 k_1+\cdots+k_r\geq r+3,
 \]
 we have 
 \[
 \sum_{m_1\geq \cdots \geq m_r\geq 3}\frac{m_r^2}{m_1^{k_1}\cdots m_r^{k_r}}<+\infty.
 \]
  \end{lemma}
 \begin{proof}
 As \[
 \sum_{m_1\geq \cdots \geq m_r\geq 3} \frac{1}{m_1^2m_2\cdots m_r}<+\infty,
 \]
 it suffices to show that
 \begin{equation}\label{eq:mix}
 \frac{m_r^2}{m_1^{k_1}\cdots m_r^{k_r}}\leq \frac{1}{m_1^2m_2\cdots m_r}, \quad m_1\geq m_2\geq \cdots\geq m_r\geq 3.  \end{equation}
 The formula \eqref{eq:mix} is equivalent to
 \begin{equation}\label{eq:or}
 m_1^{k_1-2}m_2^{k_2-1}\cdots m_{r-1}^{k_{r-1}-1}m_r^{k_r-3}\geq 1, \quad m_1\geq m_2\geq \cdots\geq m_r\geq 3. \end{equation}
 By \[ m_1\geq m_2\geq \cdots\geq m_r\geq 3   \]
 and \[
 k_1+\cdots+k_r\geq r+3,
 \]
 it follows that 
 \[
 m_1^{k_1-2}m_2^{k_2-1}\cdots m_{r-1}^{k_{r-1}-1}m_r^{k_r-3}\geq m_r^{k_1+\cdots+k_r-r-3       }\geq 1.
  \]
  As a result, the formulas \eqref{eq:mix} and \eqref{eq:or} are proved.
 \end{proof}

The derivative is positive and finite at every non-dyadic point.
\begin{lemma}\label{derivative}
For  \[y=\beta({\bf k})=\sum_{r=1}^{+\infty} \frac{1}{2^{k_1+\cdots+k_r}}, \quad {\bf k}\in \mathcal{T}\]
 and $y$ is not a dyadic rational number in $\left(0, \frac{1}{2}\right)$, we have 
 \[
     0<\Phi^\prime(y)<+\infty. \]
 \end{lemma}
\begin{proof}
By \eqref{der},  one has 
\begin{equation}\label{derconv}
\begin{split}
&\;\;\;\;\Phi^\prime(y) \\
&=\mathop{\mathrm{lim}}_{r\rightarrow+\infty}\left( \sum_{\substack{m_1\geq \cdots \geq m_r\geq 2\\ m_1=\cdots=m_r=2      }} + \sum_{\substack{m_1\geq \cdots \geq m_r\geq 2\\ m_1\geq 3, m_2=\cdots=m_r=2      }}  +\cdots+  \sum_{\substack{m_1\geq \cdots \geq m_r\geq 2\\ m_1\geq \cdots\geq m_r\geq 3     }}  \right)  \left(\frac{2}{m_1}\right)^{k_1}\cdots  \left(\frac{2}{m_r}\right)^{k_r}                             \\
&=1+\sum_{r=1}^{+\infty}  \sum_{m_1\geq \cdots \geq m_r\geq 3}      \left(\frac{2}{m_1}\right)^{k_1}\cdots  \left(\frac{2}{m_r}\right)^{k_r}.\\
\end{split}
\end{equation}
As $y=\beta(k_1,\cdots,k_r,\cdots)$ is not a dyadic rational number, we have $k_i\geq 2$ for infinitely many $i\geq 1$.
Thus there is an $N\geq 2$ such that 
\[
k_1+\cdots+k_r\geq r+3
\]
for $r\geq N$. Therefore
\[
\begin{split}
&\;\;\;\;    \sum_{r=N}^{+\infty}  \sum_{m_1\geq \cdots \geq m_r\geq 3}      \left(\frac{2}{m_1}\right)^{k_1}\cdots  \left(\frac{2}{m_r}\right)^{k_r}     \\
&= \sum_{m_1\geq \cdots \geq m_N\geq 3}      \left(\frac{2}{m_1}\right)^{k_1}\cdots  \left(\frac{2}{m_N}\right)^{k_N}    \bigg{[}1+ \sum_{m_N\geq m_{N+1}\geq 3}  \left( \frac{2}{m_{N+1}}   \right)^{k_{N+1}}+\cdots\\
&\;\;\;\;+ \sum_{m_N\geq m_{N+1}\geq \cdots\geq m_{N+s}\geq 3} \left(  \frac{2}{m_{N+1}}  \right)^{k_{N+1}}  \cdots  \left(  \frac{2}{m_{N+s}}  \right)^{k_{N+s}} +\cdots  \bigg{]}                  \\
\end{split}
\]
\[
\begin{split}
&\leq \sum_{m_1\geq \cdots \geq m_N\geq 3}      \left(\frac{2}{m_1}\right)^{k_1}\cdots  \left(\frac{2}{m_N}\right)^{k_N}    \bigg{[}1+ \sum_{m_N\geq m_{N+1}\geq 3}  \left( \frac{2}{m_{N+1}}   \right)+\cdots\\
&\;\;\;\;+ \sum_{m_N\geq m_{N+1}\geq \cdots\geq m_{N+s}\geq 3} \left(  \frac{2}{m_{N+1}}  \right) \cdots  \left(  \frac{2}{m_{N+s}}  \right)+\cdots  \bigg{]}                  \\
&\leq \sum_{m_1\geq \cdots \geq m_N\geq 3}      \left(\frac{2}{m_1}\right)^{k_1}\cdots  \left(\frac{2}{m_N}\right)^{k_N}  \prod_{m_N\geq m\geq 3} \left(1-\frac{2}{m}\right)^{-1}\\
&\leq \sum_{m_1\geq \cdots \geq m_N\geq 3}      \left(\frac{2}{m_1}\right)^{k_1}\cdots  \left(\frac{2}{m_N}\right)^{k_N}  \frac{(m_N-1)m_N }{2}\\ 
&\leq 2^{k_1+\cdots+k_N-1} \sum_{m_1\geq \cdots \geq m_N\geq 3}   \frac{m_N^2}{m_1^{k_1}\cdots m_{N-1}^{k_{N-1}} m_N^{k_N}}.\\ \end{split}
\]
By \Cref{lem:conv}, one has 
\[
     0<\Phi^\prime(y)<+\infty. \]
\end{proof}

\subsection{Evaluations at multiple zeta-star values}

The following lemma  will be used in the evaluating the normalized approximation function.
\begin{lemma}\label{lem:branch-separation}
For an admissible index
$\mathbf{k}=(k_1,\ldots,k_s)$, define 
\[
\mathbf{p}=\left(k_1,\ldots,k_{s-1},k_s+1,\{1\}^{\infty}\right).\]
For  $\mathbf{c}\in\Tset$, let $t$ be the first position at which
$\mathbf{c}$ differs from
$\mathbf{p}$.
For $$\alpha=\eta({\bf p})=\zeta^\star({\bf k}),$$ there are constants $\delta^-_{\mathbf{p}},\delta^+_{\mathbf{p}}>0$ with the
following properties.
\begin{enumerate}[label=(\alph*),leftmargin=2.4em]
\item If $t\le s$ and $c_t>p_t$, then
$\eta(\mathbf{c})\le\alpha-\delta^-_{\mathbf{p}}$.
\item If $t<s$ and $c_t<p_t$, or if $t=s$ and $c_s\le p_s-2$, then
$\eta(\mathbf{c})\ge\alpha+\delta^+_{\mathbf{p}}$.
\item The only upper branch whose closure contains $\alpha$ is
$t=s$, $c_s=p_s-1$; its coordinates are precisely those beginning with
$$(p_1,\ldots,p_{s-1},p_s-1)=(k_1,\cdots,k_{s-1},k_s).$$
\end{enumerate}
\end{lemma}

\begin{proof}
Denote by 
\[
 P_j:=p_1+\cdots+p_j,\quad j=1,\cdots, s.
\]
Suppose first that $c_t>p_t$.  Among all such points, the largest binary
coordinate is approached when $c_t=p_t+1$ and all later digits are $1$.  Its
contribution from the $t$-th digit onward is
\[
 \sum_{n=P_t+1}^{\infty}2^{-n}=2^{-P_t},
\]
which accounts only for the $t$-th binary $1$ of $\mathbf{p}$ and omits the
strictly positive tail after it.  This gives a positive gap on the lower side.

Now suppose that $c_t<p_t$.  The closest point on the upper side is obtained
in the closure by taking $c_t=p_t-1$ and sending the next digit to infinity.
Its contribution from position $t$ onward is then $2^{-(P_t-1)}$.  On the
other hand,
\[
 \sum_{j\ge t}2^{-P_j}\le\sum_{n=P_t}^{\infty}2^{-n}
 =2^{-(P_t-1)},
\]
and equality holds precisely when every digit of $\mathbf{p}$ after position
$t$ is $1$.  For $t<s$ this is impossible because $p_s=k_s+1\ge2$; for
$t=s$ it holds and gives the condition $c_s=p_s-1=k_s$.  If instead
$c_s\le p_s-2$, the first binary $1$ moves still farther
left, so the gap is again strict.

There are only finitely many relevant positions $t\le s$. Taking the minimum of their positive binary gaps,
and then applying the continuous strictly increasing map $\Phi$, gives
$\delta^-_{\mathbf p}$ and $\delta^+_{\mathbf p}$.
\end{proof}

\begin{theorem}
\label{thm:exact-finite}
For every finite admissible index $\mathbf{k}=(k_1,\cdots,k_s)$,
\begin{equation}\label{eq:exact-finite}
 \mathcal{N}\bigl(\zeta^\star(\mathbf{k})\bigr)
 =\Lnorm(\mathbf{k})
 =2^{\wt(\mathbf{k})}
 \sum_{n_1\ge\cdots\ge n_s\ge2}
 \frac{n_s-1}{n_1^{k_1}\cdots n_s^{k_s}}.
\end{equation}
\end{theorem}
\begin{proof}
 For $d\geq 1$, let $r=s+d$, and let $\mathbf{b}$ be any
admissible index of depth $r$.  The infinite sequence of
$\eta^{-1}\left(\zeta^\star(\mathbf{b})\right)$ is
\[
 \widehat{\mathbf{b}}=\left(\widehat{b}_i\right)_{i\geq 1}
 =(b_1,\ldots,b_{r-1},b_r+1,\{1\}^{\infty}).
\]
Denote by $\alpha=\zeta^\star\left({\bf k}\right)$. 
Define \[
{\bf p}=\eta^{-1}(\alpha)= \left(k_1,\cdots, k_{s-1}, k_s+1,\{1\}^{\infty}\right).
\]
If the first difference between $\widehat{\mathbf{b}}$ and $\mathbf{p}$ occurs
within the first $s$ positions and is not the adjacent upper branch, then
\cref{lem:branch-separation} and
$\wt(\mathbf{b})\ge r+1$ give a lower bound of the form
\begin{equation}\label{eq:separated-growth}
 2^{\wt(\mathbf{b})}
 \bigl|\alpha-\zeta^\star(\mathbf{b})\bigr|
 \ge 2^{r+1}\delta_{\mathbf{p}},
\end{equation}
for a fixed $\delta_{\mathbf{p}}>0$.

Suppose next that the first difference occurs at $t=s+e$ with
$1\le e\le d$.  Since the tail of $\mathbf p$ consists of $1$'s, one has
$\widehat b_t\ge2$, and the order structure gives
$\zeta^\star(\mathbf b)<\alpha$.  The largest value in this branch is the
real number $q_e$ defined in \eqref{eq:qd}, so
$\alpha-\zeta^\star(\mathbf b)\ge D_e$.  Moreover,
$ b_i=p_i$ for $i<t$, and
$\sum_{i=1}^{r} \widehat{b}_i=\wt(\mathbf b)+1$.  
Denote by $K=k_1+\cdots+k_s$.
Consequently,
\[
 \wt(\mathbf b)+1
 \ge (K+1)+(e-1)+2+(d-e)=K+d+2,
\]
and hence
\[
 2^{\wt(\mathbf{b})}
 \bigl(\alpha-\zeta^\star(\mathbf{b})\bigr)
 \ge 2^{K+d+1}D_e
 =2^{d-e}B_e\ge B_d,
\]
where the last inequality follows by iterating $B_{j+1}<2B_j$.
The real number $q_d$ itself has depth $s+d$ and normalized error exactly
$B_d$.

Finally, since $d\ge1$, an adjacent upper bound has the form
\[
 \mathbf{b}=(k_1,\ldots,k_s,c_1,\ldots,c_d),
 \qquad c_i\ge1.
\]
Retaining only the positive subseries with
$n_{s+1}=2$ and $n_{s+2}=\cdots=n_{s+d}=1$ gives
\[
 \zeta^\star(\mathbf{b})-\alpha
 \ge 2^{-c_1}
 \sum_{n_1\ge\cdots\ge n_s\ge2}
 \frac{1}{n_1^{k_1}\cdots n_s^{k_s}}.
\]
Therefore
\begin{equation}\label{eq:adjacent-upper-growth}
 2^{\wt(\mathbf{b})}
 \bigl(\zeta^\star({\mathbf{b}})-\alpha\bigr)
 \ge C_{\mathbf{k}}2^{d-1},
\end{equation}
where
\[
 C_{\mathbf{k}}
 :=2^K\sum_{n_1\ge\cdots\ge n_s\ge2}
 \frac{1}{n_1^{k_1}\cdots n_s^{k_s}}>0.
\]

The preceding alternatives are exhaustive: the first mismatch occurs either
within the first $s$ positions, after position $s$, or in the unique adjacent
upper branch described in \cref{lem:branch-separation}.

If $\Lnorm(\mathbf{k})<+\infty$, then $B_d$ increases to
$\Lnorm(\mathbf{k})$, while the bounds
\eqref{eq:separated-growth} and \eqref{eq:adjacent-upper-growth} tend to
$+\infty$.  Hence
\[
 \Delta_{s+d}(\alpha)=B_d
\]
for all sufficiently large $d$.  If $\Lnorm(\mathbf{k})=+\infty$, all three
potential bounds tend to $+\infty$, including the adjacent lower bound
$B_d$, so $\Delta_{s+d}(\alpha)\to+\infty$.  In both cases,
\[
 \mathcal{N}(\alpha)=\lim_{d\to\infty}B_d=\Lnorm(\mathbf{k}).
\]
\end{proof}

\begin{corollary}\label{cor:spikes}
For every finite admissible index $\mathbf{k}$,
\[
 \mathcal{N}\bigl(\zeta^\star({\mathbf{k}})\bigr)>1
 \quad\text{or}\quad
 \mathcal{N}\bigl(\zeta^\star({\mathbf{k}})\bigr)=+\infty.
\]
The second alternative occurs exactly when
$\mathbf{k}=(2,\ones{s-1})$.
\end{corollary}

\begin{proof}
In \eqref{eq:exact-finite}, the term
$n_1=\cdots=n_s=2$ contributes exactly $1$ after multiplication by
$2^{\wt(\mathbf{k})}$, and there are further positive terms.  The infinite
case is characterized by \eqref{eq:infinite-interval}.
\end{proof}

\begin{theorem}\label{thm:nowhere-continuous}
The map $\alpha\mapsto\mathcal{N}(\alpha)$ from $(1,+\infty)$ to the extended
half-line $[0,+\infty]$, endowed with its order topology, is nowhere
continuous.
\end{theorem}

\begin{proof}
Every nonempty open interval contains a point where
$\mathcal{N}(\alpha)=0$ by \cref{thm:residual-zero}, and it contains a 
multiple zeta-star value by density of $\mathcal{Z}^\star$.  At the latter point,
\cref{cor:spikes} gives a value strictly larger than $1$ or equal to
$+\infty$.  Under the compactifying homeomorphism
$\psi(t)=t/(1+t)$ and $\psi(+\infty)=1$, the image of every neighborhood has
oscillation at least $1/2$.  Thus continuity fails at every point.
\end{proof}

\begin{theorem}
For every $\alpha>1$,
\[
\mathcal N(\alpha)=+\infty
\qquad\Leftrightarrow\qquad
\alpha=\zeta^\star(2,\ones{s-1})\text{ for some }s\ge1.
\]
\end{theorem}
\begin{proof}
By \Cref{cor:spikes}, it suffices to show that 
\[
\mathcal{N}(\alpha)<+\infty
\]
for $\alpha\notin \mathcal{Z}^\star$. 
For $\alpha=\Phi(y)\notin \mathcal{Z}^\star$, $y$ is not a dyadic rational number.
 By \Cref{derivative}, 
\[
\Phi^\prime(y)<+\infty.
\]
Since $\Lambda(y)\leq 1$, 
by \Cref{prop:differentiable}, we have 
\[
\mathcal{N}(\Phi(y))=\Phi^\prime (y) \Lambda(y)\leq \Phi^\prime(y)<+\infty\]
for $\alpha=\Phi(y)\notin \mathcal{Z}^\star$. \end{proof}

\subsection{Examples}
For $k\ge3$, the upper endpoint of the basic interval for the one-term index
$(k)$ is $\zeta(k-1)$.  Hence
\begin{equation}\label{eq:zeta-examples}
 \mathcal{N}(\zeta(2))=+\infty,
 \qquad
 \mathcal{N}(\zeta(k))
 =2^k\bigl(\zeta(k-1)-\zeta(k)\bigr)\quad(k\ge3).
\end{equation}
Since
\[
 \zeta(k-1)-\zeta(k)
 =2^{-k}+O(3^{-k}),
\]
one obtains
\begin{equation}\label{eq:zeta-limit}
 \mathcal{N}(\zeta(k))
 =1+O\bigl((2/3)^k\bigr)\longrightarrow1^+.
\end{equation}
Thus $1\in \overline{\mathrm{Im}\;\mathcal{N}}$, although no multiple
zeta-star point has normalized value $1$.

Euler's identity $\zeta(3,1)=\tfrac14\zeta(4)$ \cite{HoffmanMHS} gives
$\zeta^\star(3,1)=\tfrac54\zeta(4)$, and therefore
\[
 \mathcal{N}\bigl(\zeta^\star(3,1)\bigr)
 =16\left(\zeta(2)-\frac54\zeta(4)\right)
 \approx4.6724803953.
\]

Fix an integer $n\ge2$ and define the block
\begin{equation}\label{eq:block-bn}
 \mathbf{b}_n:=(2,\ones{n-2}),
\end{equation}
where the sequence of $1$'s is empty when $n=2$.
By the Ohno-Wakabayashi cyclic sum formula in Theorem $1$ of \cite{OW}, 
\begin{equation}\label{eq:cyclic-spec}
 \zeta^\star(\{\mathbf{b}_n\}^{a},1)
 =n\,\zeta(an+1).
\end{equation}
The formula \eqref{eq:cyclic-spec} can be found in Examples (b) in \cite{OW}.

Set
\begin{equation}\label{eq:Qn-xn}
 Q_n:=2^n-1,
 \qquad
 x_n:=\sum_{q=0}^{\infty}\sum_{t=2}^{n}2^{-(qn+t)}
 =\frac{2^{n-1}-1}{2^n-1}
 =\frac{1}{2}-\frac1{2Q_n}.
\end{equation}
Thus the binary expansion of $x_n$ is
\[
 (0\underbrace{1\cdots1}_{n-1})^{\infty}.
\]
For $a\ge1$, define
\[
 \mathbf{K}_{n,a}:=(\{\mathbf{b}_n\}^{a},1)
\]
and
\begin{equation}\label{eq:xna}
 x_{n,a}:=\beta^f(\mathbf{K}_{n,a})
 =\sum_{q=0}^{a-1}\sum_{t=2}^{n}2^{-(qn+t)}+2^{-an-1}.
\end{equation}
The tail of $x_n$ after $a$ periods is $2^{-an}x_n$, so
\begin{equation}\label{eq:xna-distance}
 x_{n,a}-x_n
 =2^{-an}\left(\frac{1}{2}-x_n\right)
 =\frac{2^{-an}}{2Q_n}.
\end{equation}
By \cref{prop:finite-coding} and \eqref{eq:cyclic-spec},
\[
 \Phi(x_{n,a})=n\zeta(an+1).
\]
Letting $a\to+\infty$ and using continuity of $\Phi$, one has
\begin{equation}\label{eq:integer-coded}
 \Phi(x_n)=n.
\end{equation}
The denominator $Q_n$ is odd and greater than $1$, so $x_n$ is non-dyadic.
By Theorem $1.7$ in \cite{HMOInfinite}, the one-sided derivatives of $\Phi$
therefore agree at $x_n$.  Since
$x_{n,a}>x_n$, 
by \eqref{eq:xna-distance} we have
\begin{align}
 \Phi'(x_n)
 &=\lim_{a\to+\infty}
 \frac{\Phi(x_{n,a})-\Phi(x_n)}{x_{n,a}-x_n}\notag\\
 &=\lim_{a\to+\infty}
 2nQ_n2^{an}(\zeta(an+1)-1)
 =nQ_n.\label{eq:integer-derivative}
\end{align}

The candidate $\mathbf{K}_{n,a}$ has depth and weight
\begin{equation}\label{eq:candidate-depth-weight}
 \dep(\mathbf{K}_{n,a})=a(n-1)+1,
 \qquad
 \wt(\mathbf{K}_{n,a})=an+1.
\end{equation}
It yields the upper bound
\begin{align}
 \Delta_{a(n-1)+1}(n)
 &\le2^{an+1}
 \left|n-\zeta^\star(\{\mathbf{b}_n\}^{a},1)\right|\notag\\
 &=n2^{an+1}(\zeta(an+1)-1)
 \longrightarrow n.\label{eq:integer-upper}
\end{align}
Thus $\mathcal{N}(n)\le n$.

The lower bound is arithmetic.

\begin{lemma}\label{lem:sep}
Let $\mathbf{k}$ be any finite admissible index of weight $w$.  Then
\begin{equation}\label{eq:sep}
 2^w\bigl|\beta^f(\mathbf{k})-x_n\bigr|
 \ge\frac{1}{Q_n}.
\end{equation}
\end{lemma}

\begin{proof}
By \cref{prop:finite-coding},
$\beta^f(\mathbf{k})=A/2^{w}$ for an odd integer $A$.  Hence
\[
 2^w\left|\frac{A}{2^{w}}-\frac{Q_n-1}{2Q_n}\right|
 =\frac{\left|AQ_n-(Q_n-1)2^{w-1}\right|}{Q_n}.
\]
The integer inside the absolute value is odd: $AQ_n$ is odd and the other term
is even.  Its absolute value is therefore at least $1$.
\end{proof}

\begin{proposition}
\label{prop:integer-lower-sequence}
Let $\mathbf{k}^{(j)}$ be any sequence of finite admissible indices whose
depths tend to infinity.  Then
\begin{equation}\label{eq:integer-sequence-lower}
 \liminf_{j\to+\infty}
 2^{\wt(\mathbf{k}^{(j)})}
 \left|n-\zeta^\star(\mathbf{k}^{(j)})\right|
 \ge n.
\end{equation}
\end{proposition}

\begin{proof}
Denote by
\[
 x_j=\beta^f(\mathbf{k}^{(j)}),
 \qquad
 w_j=\wt(\mathbf{k}^{(j)}),
 \qquad
 E_j=2^{w_j}|\Phi(x_j)-\Phi(x_n)|.
\]
If $\liminf E_j=+\infty$, the claim is immediate.  Otherwise pass to a
subsequence on which $E_j$ converges to its finite lower limit.  Since
$w_j\ge\dep(\mathbf{k}^{(j)})+1\to+\infty$, boundedness of $E_j$ gives
$\Phi(x_j)\to\Phi(x_n)$.  The inverse of the continuous increasing bijection
$\Phi$ is continuous, so $x_j\to x_n$.  Each $x_j$ is dyadic and $x_n$ is not;
hence $x_j\ne x_n$.  Using \eqref{eq:integer-derivative} and
\cref{lem:sep},
\begin{align*}
 \liminf_{j\to\infty}E_j
 &=\liminf_{j\to\infty}
 \frac{|\Phi(x_j)-\Phi(x_n)|}{|x_j-x_n|}
 2^{w_j}|x_j-x_n|\\
 &\ge nQ_n\cdot\frac{1}{Q_n}=n.
\end{align*}
\end{proof}

\begin{theorem}\label{thm:integer-values}
For every integer $n\ge2$,
\begin{equation}\label{eq:integer-values}
 \mathcal{N}(n)=n.
\end{equation}
\end{theorem}

\begin{proof}
The upper bound follows from \eqref{eq:integer-upper} and in particular shows
that $\mathcal{N}(n)<+\infty$.  For the reverse bound, choose depths
$r_j\to+\infty$ with
$\Delta_{r_j}(n)\to\mathcal{N}(n)$ and choose
$\mathbf{k}^{(j)}\in\Sset_{r_j}$ such that
\[
 2^{\wt(\mathbf{k}^{(j)})}
 \bigl|n-\zeta^\star(\mathbf{k}^{(j)})\bigr|
 \le\Delta_{r_j}(n)+\frac1j.
\]
Then \cref{prop:integer-lower-sequence} gives
$n\le\mathcal{N}(n)$.  Combining the two inequalities proves
\eqref{eq:integer-values}.
\end{proof}

The upper-bound sequence saturates the dyadic obstruction exactly:
\[
 2^{\wt(\mathbf{K}_{n,a})}(x_{n,a}-x_n)=\frac{1}{Q_n}.
\]
The derivative $nQ_n$ converts this coordinate-scale constant into the
normalized value $n$.

\section{The closure of the image}\label{sec:image}

\subsection{Derivative formulas  and related upper bounds}

Set
\begin{equation}\label{eq:As}
 A_s:=\sum_{n_1\ge\cdots\ge n_s\ge3}
 \frac{1}{n_1^3n_2\cdots n_s},
 \qquad s\ge1.
\end{equation}
\begin{lemma}\label{estimate}
There is an absolute constant $C_A>0$ such that, for every $s\ge1$,
\begin{equation}\label{eq:As-estimate}
 0\le2^{-(s+1)}-A_s\le C_A\frac{s}{3^s}.
\end{equation}
\end{lemma}
\begin{proof} 
For $n\ge3$, put
\[
 P_n(z):=\prod_{m=3}^n\left(1-\frac{z}{m}\right)^{-1}.
\]
Expanding every factor as a geometric series shows that
\[
 [z^{s-1}]P_n(z)
 =\sum_{n\ge n_2\ge\cdots\ge n_s\ge3}
   \frac{1}{n_2\cdots n_s}.
\]
Consequently, for $|z|<1$,
\begin{equation}\label{eq:F-def}
 F(z):=\sum_{s\ge1}A_s z^{s-1}
 =\sum_{n=3}^{\infty}\frac{P_n(z)}{n^3}.
\end{equation}
The convergence is locally uniform in $|z|<1$.

We compare this with
\[
 G(z):=\sum_{n=3}^{\infty}
 \frac{P_n(z)}{n(n-1)(n-2)}.
\]
Using the gamma and beta functions, we have
\[
 P_n(z)
 =\frac{\Gamma(n+1)\Gamma(3-z)}{2\Gamma(n+1-z)}
\]
and hence
\[
 \frac{P_n(z)}{n(n-1)(n-2)}
 =\frac{1}{2}B(n-2,3-z).
\]
Thus, for $0<z<1$, we have
\begin{align*}
 G(z)
 &=\frac12\sum_{n=3}^{\infty}
   \int_0^1 x^{n-3}(1-x)^{2-z}\,dx \\
 &=\frac12\int_0^1(1-x)^{2-z}
   \sum_{n=3}^{\infty}x^{n-3}\,dx \\
 &=\frac12\int_0^1(1-x)^{1-z}\,dx
 =\frac{1}{2(2-z)}.
\end{align*}
Therefore
\begin{equation}\label{eq:G-coeff}
 [z^{s-1}]G(z)=2^{-(s+1)}.
\end{equation}

Let
\[
 \delta_n:=\frac{1}{n(n-1)(n-2)}-\frac{1}{n^3}
 =\frac{3n-2}{n^3(n-1)(n-2)}.
\]
Since $\delta_n\ge0$ and every coefficient of $P_n(z)$ is nonnegative,
\eqref{eq:F-def} and \eqref{eq:G-coeff} imply
\[
 2^{-(s+1)}-A_s
 =[z^{s-1}]\sum_{n=3}^{\infty}\delta_nP_n(z)\ge0.
\]

It remains to estimate this difference. For $n\ge4$,
\begin{equation}\label{eq:delta-bound}
 \delta_n\le
 \frac{3}{n(n-1)(n-2)(n-3)},
\end{equation}
indeed, after multiplication by the positive denominator,
\eqref{eq:delta-bound} is equivalent to
$(3n-2)(n-3)\le3n^2$.
Define
\[
 H(z):=\sum_{n=4}^{\infty}
 \frac{P_n(z)}{n(n-1)(n-2)(n-3)}.
\]
As above,
\[
 \frac{P_n(z)}{n(n-1)(n-2)(n-3)}
 =\frac{1}{2(3-z)}\,B(n-3,4-z),
\]
so, for $\operatorname{Re}z<3$,
\begin{align*}
 H(z)
 &=\frac{1}{2(3-z)}\sum_{n=4}^{\infty}
   \int_0^1x^{n-4}(1-x)^{3-z}\,dx \\
 &=\frac{1}{2(3-z)}\int_0^1(1-x)^{2-z}\,dx
 =\frac{1}{2(3-z)^2}.
\end{align*}
In particular,
\begin{equation}\label{eq:H-coeff}
 [z^{s-1}]H(z)=\frac{s}{2\,3^{s+1}}.
\end{equation}
For the remaining term $n=3$, we have
\[
 \delta_3=\frac{7}{54},
 \qquad
 P_3(z)=\left(1-\frac z3\right)^{-1},
\]
and therefore
\begin{equation}\label{eq:n3-coeff}
 [z^{s-1}]\bigl(\delta_3P_3(z)\bigr)
 =\frac{7}{18}\frac{1}{3^s}.
\end{equation}
Combining \eqref{eq:delta-bound}--\eqref{eq:n3-coeff}, we obtain
\begin{align*}
 2^{-(s+1)}-A_s
 &\le \frac{7}{18}\frac{1}{3^s}
   +3\,[z^{s-1}]H(z) \\
 &=\frac{7}{18}\frac{1}{3^s}
   +\frac12\frac{s}{3^s} \\
 &\le \frac89\frac{s}{3^s},
\end{align*}
where the last inequality uses $s\ge1$. This proves the lemma with
$C_A=8/9$.  
\end{proof}

\begin{lemma}\label{lem:zero-count}
For $k_1\geq 3, k_2,\cdots,k_r\geq 1$, one has 
\begin{equation}\label{eq:zero-count}
0\leq \sum_{m_1\geq \cdots \geq m_r\geq 3}      \left(\frac{2}{m_1}\right)^{k_1}\cdots  \left(\frac{2}{m_r}\right)^{k_r}< 5\left( \frac{2}{3}\right)^{k_1+\cdots+k_r-r}.
\end{equation}
\end{lemma}

\begin{proof}
For  $k_1\geq 3, k_2,\cdots,k_r\geq 1$ and 
\[
m_1\geq \cdots \geq m_r\geq 3,\]
it is clear that
\[
\left(   \frac{2}{m_1} \right)^{k_1}\leq \left(\frac{2}{3}   \right)^{k_1-3}\left(   \frac{2}{m_1} \right)^{3}\]
and 
\[
\left(   \frac{2}{m_i} \right)^{k_i}\leq \left(\frac{2}{3}   \right)^{k_i-1}\left(   \frac{2}{m_i} \right),\quad  2\leq i\leq r.\]
Hence
\[
 \left(\frac{2}{m_1}\right)^{k_1}\cdots  \left(\frac{2}{m_r}\right)^{k_r}\leq \left(\frac{2}{3}\right)^{k_1+\cdots+k_r-r-2       } \left(   \frac{2}{m_1} \right)^{3} \prod_{i=2}^r\left(   \frac{2}{m_i} \right).
 \]
 By \Cref{estimate},  
 \[
   \sum_{  m_1\geq \cdots \geq m_r\geq 3    }   \left(   \frac{2}{m_1} \right)^{3} \prod_{i=2}^r\left(   \frac{2}{m_i} \right) =2^{r+2} \sum_{  m_1\geq \cdots \geq m_r\geq 3    }  \frac{1}{m_1^3m_2\cdots m_r} \leq 2. \]
   Therefore 
   \[
 \sum_{m_1\geq \cdots \geq m_r\geq 3}      \left(\frac{2}{m_1}\right)^{k_1}\cdots  \left(\frac{2}{m_r}\right)^{k_r}\leq  \frac{9}{2}\left( \frac{2}{3}\right)^{k_1+\cdots+k_r-r}<5 \left( \frac{2}{3}\right)^{k_1+\cdots+k_r-r}.\]
\end{proof}

For $q\ge2$, define
\begin{equation}\label{eq:xq-Sq}
 x_q:=\frac{1}{4}-2^{-q-2}=\sum_{j=2}^{q+1}2^{-j-1},
 \qquad
 S_q:=\partial_+\Phi(x_q).
\end{equation}

\begin{proposition}\label{prop:linear-slopes}
There is a real constant $c_0$ such that
\begin{equation}\label{eq:Sq-asymptotic}
 S_q=2q+c_0+O\!\left(q\left(\frac23\right)^q\right).
\end{equation}
In particular, $S_q=2q+O(1)$.
\end{proposition}

\begin{proof}
By \eqref{eq:fi}, we have 
\[
\begin{split}
 &\;\;\;\; S_q\\
 &=\partial_+\Phi(x_q)\\
 &=\sum_{n_1\geq \cdots \geq n_q\geq 2}\left(\frac{2}{n_1}   \right)^{3}\left(\frac{2}{n_2}   \right)  \cdots   \left(\frac{2}{n_q}   \right)\\
 &= \left(  \sum_{\substack{n_1\geq \cdots \geq n_q\geq 2\\ n_1=\cdots=n_q=2      }} +  \sum_{\substack{n_1\geq \cdots \geq n_q\geq 2\\ n_1\geq 3,n_2=\cdots= n_q=2      }}+\cdots         + \sum_{\substack{n_1\geq \cdots \geq n_q\geq 2\\ n_1\geq \cdots\geq n_q\geq 3     }}    \right)   \left(\frac{2}{n_1}   \right)^{3}\left(\frac{2}{n_2}   \right)  \cdots   \left(\frac{2}{n_q}   \right)   \\
 &=1+\sum_{n_1\geq 3} \left( \frac{2}{n_1}    \right)^3+\cdots+ \sum_{n_1\geq \cdots\geq n_{q-1}\geq 3} \left( \frac{2}{n_1}    \right)^3\left(\frac{2}{n_2}   \right)  \cdots   \left(\frac{2}{n_{q-1}}   \right) \\
 &\;\;\;\; + \sum_{n_1\geq \cdots\geq n_{q}\geq 3} \left( \frac{2}{n_1}    \right)^3\left(\frac{2}{n_2}   \right)  \cdots   \left(\frac{2}{n_q}   \right)\\
 &=1+2^3\sum_{n_1\geq 3}  \frac{1}{n_1^3}+\cdots+ 2^{q+1} \sum_{n_1\geq \cdots\geq n_{q-1}\geq 3} \frac{1}{n_1^3n_2\cdots n_{q-1}} +2^{q+2}   \sum_{n_1\geq \cdots\geq n_{q}\geq 3} \frac{1}{n_1^3n_2\cdots n_q}.\\    \end{split}
\]
From \Cref{estimate}, it follows that
\[
S_q=1+\sum_{s=1}^q2^{s+2} A_s=2q+1+\sum_{s=1}^q 2\left( 2^{s+1}A_s-1\right)=2q+c_0+O\left(q\left( \frac{2}{3} \right)^q  \right).
\]
Here $$c_0=1+\sum_{s=1}^{+\infty} 2\left( 2^{s+1}A_s-1    \right).$$
\end{proof}

\subsection{Periodic tails and dense image}

For $m\ge2$, let
\begin{equation}\label{eq:theta-m}
 \theta_m:=\frac{1}{2^m-1}.
\end{equation}
\begin{lemma}\label{lem:theta-lambda}
For every $m\ge2$,
\begin{equation}\label{eq:theta-lambda}
 \Lambda(\theta_m)=\frac{1}{2^m-1}.
\end{equation}
\end{lemma}

\begin{proof}
Modulo $1$, the doubling orbit of $\theta_m$ is periodic and consists of the
residues $2^j/(2^m-1)$, $0\le j\le m-1$.  Its smallest distance to an integer
is $1/(2^m-1)$.  Apply \cref{lem:orbit-formula}.
\end{proof}

For $q,m\ge2$ and $L\ge1$, put
\begin{equation}\label{eq:yqml}
 y_{q,m,L}:=x_q+2^{-(q+2+L)}\theta_m.
\end{equation}
Its binary expansion consists of the terminating prefix of $x_q$, followed
by exactly $L$ zero digits and then the complete periodic binary expansion of
$\theta_m$.

\begin{proposition}
\label{prop:pf}
For fixed $q,m\ge2$,
\begin{align}
 \Lambda(y_{q,m,L})&=\frac{1}{2^m-1},\label{eq:y-lambda}\\
 \Phi'(y_{q,m,L})&\longrightarrow S_q
 \qquad(L\to+\infty).\label{eq:y-derivative}
\end{align}
Consequently,
\begin{equation}\label{eq:grid-in-closure}
 \frac{S_q}{2^m-1}
 \in\overline{\{\mathcal{N}(\alpha):\alpha>1,
 \ \mathcal{N}(\alpha)<+\infty\}}.
\end{equation}
\end{proposition}

\begin{proof}
The number $y_{q,m,L}$ is obtained from $\theta_m$ by adding a finite binary
prefix, so \cref{cor:inva,lem:theta-lambda} give
\eqref{eq:y-lambda}.
The derivative formula applies because $y_{q,m,L}$ has infinitely many zeros
and ones. For 
\[
y_{q,m,L}:=x_q+2^{-(q+2+L)}\theta_m,\quad \theta_m=\left(2^m-1\right)^{-1},
\]
it is clear that 
\[
y_{q,m,L}=\beta\left(3,\{1\}^{q-1},L+m,\{m\}^{\infty}   \right).
\]
Consequently, if $r=q+j$, then
\[
K_r-r=L+2+j(m-1).
\]
By
\Cref{lem:zero-count}, therefore
\[
 0\le\Phi'(y_{q,m,L})-S_q
 \le 5 \sum_{j=1}^{+\infty}
 \left(\frac23\right)^{L+2+j(m-1)},
\]
which tends to zero with $L\rightarrow+\infty$.  Finally,
\cref{prop:differentiable} gives
\[
 \mathcal{N}(\Phi(y_{q,m,L}))
 =\Phi'(y_{q,m,L})\frac{1}{2^m-1}.
\]
Letting $L\to+\infty$ proves \eqref{eq:grid-in-closure}.
\end{proof}

\begin{theorem}\label{thm:image-closure}
One has
\begin{equation}\label{eq:image-closure}
 \overline{\{\mathcal{N}(\alpha):\alpha>1,
 \ \mathcal{N}(\alpha)<+\infty\}}
 =[0,+\infty).
\end{equation}
Moreover, $0$ and $+\infty$ are actual values of
$\alpha\mapsto\mathcal{N}(\alpha)$.
\end{theorem}

\begin{proof}
The value $0$ is attained because \cref{thm:residual-zero} gives a nonempty
zero set.  The value $+\infty$ is attained at $\zeta(2)$ by
\eqref{eq:zeta-examples}.
Let $C$ be the closed set on the left of \eqref{eq:image-closure}.  By
\cref{prop:pf}, it contains
\[
 \frac{S_q}{2^m-1}
 \qquad(q,m\ge2).
\]
Fix $t>0$.  As $m\to+\infty$, choose an integer $q_m$ with
\[
 \left|q_m-\frac{t}{2}(2^m-1)\right|\le\frac12.
\]
Using $S_q=2q+O(1)$,
\[
 \frac{S_{q_m}}{2^m-1}
 =\frac{2q_m}{2^m-1}+O(2^{-m})\to t.
\]
Hence $t\in C$.  Taking a fixed $q$ and letting $m\to+\infty$ shows that
$0\in C$.  Nonnegativity gives \eqref{eq:image-closure}.
\end{proof}

\section{Summary and further  questions}\label{sec:conclusion}

The main results reveal some very different properties.
The generic value $0$ coexists with a dense countable set of evaluations  and with explicit positive values at every integer.  This explains the nowhere continuity established in \cref{thm:nowhere-continuous}.  The image
closure theorem is stronger than the existence of many isolated examples but
weaker than a complete description of the image itself.

The arguments assembled here leave several natural questions unresolved.
\begin{enumerate}[leftmargin=2.3em]
\item Improve the linear factor in \eqref{eq:linear-cylinder}; a uniform or
sublinear bound on significant classes of sequences would strengthen the metric
theory.
\item Determine the exact finite-valued image
\[
 \{\mathcal{N}(\alpha):\alpha>1,\ \mathcal{N}(\alpha)<+\infty\},
\]
not merely its closure.  In particular, the arguments above do not decide
whether every $t\ge0$ is attained.
\item Express $\mathcal{N}(\eta({\bf k}))$ directly from  general
infinite sequences ${\bf k}$, including rational numbers or
algebraic numbers.
\item Determine Hausdorff dimensions and multifractal spectra of level sets
such as $\{\alpha:\mathcal{N}(\alpha)=t\}$ and
$\{\alpha:\mathcal{N}(\alpha)\ge t\}$.
\end{enumerate}

\section*{Acknowledgements}
This project is funded by the National Natural Science Foundation of China
  (Grant No.12571009) and the Natural Science Foundation of Hunan
  Province, China (Grant No.2026JJ40003).

\end{document}